\documentclass[10pt,reqno]{article}
\usepackage{a4wide}

\usepackage{amssymb}

\usepackage{amssymb, amsfonts, amsbsy, latexsym}
\usepackage{graphicx}
\usepackage{tikz}
\usepackage{color}
\usepackage{colortbl}
\usepackage[english]{babel}
\usepackage{url}
\usepackage{amsmath}

\usepackage{scalerel}

\usepackage{relsize}

\usepackage{amsthm}

\usepackage{bbm}

\usepackage{blkarray}

\usepackage{bm}

\newtheorem{theorem}{Theorem}
\newtheorem{definition}[theorem]{Definition}
\newtheorem{proposition}[theorem]{Proposition}
\newtheorem{remark}[theorem]{Remark}
\newtheorem{lemma}[theorem]{Lemma}
\newtheorem{example}[theorem]{Example}
\newtheorem{assumption}[theorem]{Assumption}

\newtheorem{approximation problem}[theorem]{Approximation problem}

\newtheorem{corollary}[theorem]{Corollary}

\newcommand{\ip}[2]{\left\langle#1,#2\right\rangle}
\newcommand{\abs}[1]{\left|#1\right|}
\newcommand{\norm}[1]{\left\|#1\right\|}

\def\hf{\hat{f}}
\def\hs{\hat{s}}

\def\hb{\hat{b}}
\def\hq{\hat{q}}

\def\w{\omega}
\def\d{\delta}
\def\e{\epsilon}

\def\l{\lambda}
\def\k{\kappa}
\def\a{\alpha}
\def\b{\beta}
\def\t{\theta}

\def\NN{\mathbb{N}}

\def\cF{\mathcal{F}}

\def\cH{\mathcal{H}}

\def\CC{\mathbb{C}}
\def\NN{\mathbb{N}}
\def\ZZ{\mathbb{Z}}

\def\RR{\mathbb{R}}

\begin{document}

\title{Randomized Continuous Frames in Time-Frequency Analysis
\\
$ $
\\
{\large  \ \ \ \ \ \textbf{Ron Levie} \quad \quad \quad \quad \quad \quad \quad\ \quad \ \  \textbf{Haim Avron}} \quad \\
{\normalsize   \quad \quad \textit{levie@math.lmu.de} \ \quad \quad  \quad \quad \ \ \ \ \ \quad \ \ \ \ \textit{haimav@tauex.tau.ac.il} }\\
{\normalsize  \textit{Technische Universit\"{a}t Berlin} \ \ \ \ \  \quad \quad \quad  \quad \quad  \textit{Tel Aviv University} }
}

\date{}

\maketitle

\begin{abstract}
Recently, a Monte Carlo approach was proposed for processing highly redundant continuous frames. In this paper we present and analyze applications of this new theory. The computational complexity of the Monte Carlo method relies on the continuous frame being so called \emph{linear volume discretizable} (LVD). The LVD property means that the number of samples in the coefficient space required by the Monte Carlo method is proportional to the resolution of the discrete signal. We show in this paper that the continuous wavelet transform (CWT) and the localizing time-frequency transform (LTFT) are LVD. The LTFT is a time-frequency representation based on a 3D time-frequency space with a richer class of time-frequency atoms than classical time-frequency transforms like the short time Fourier transform (STFT) and the CWT. Our analysis proves that performing signal processing with the LTFT has the same asymptotic complexity  as signal processing with the STFT and CWT (based on FFT), even though the coefficient space of the LTFT is higher dimensional.
\end{abstract}

\section{Introduction}

Continuous frames \cite{Cframe0,Cframe1} are the continuous counterpart of discrete frames, where the index of the frame elements lies in a continuous space. Two examples of continuous frames are the short time Fourier transform (STFT) \cite{Time_freq} and the continuous wavelet transform (CWT) \cite{Cont_wavelet_original,Ten_lectures}, for which the frame elements are parameterized by a 2D continuous space (time-frequency for STFT and time-scale for CWT). 
Let $f:G\rightarrow \cH$ be a continuous frame, where $G$ is the continuous index set, also called \emph{phase space}, and $\cH$ is the Hilbert space of signals. The \emph{frame analysis operator} reads
\begin{equation}
\cH\ni s \mapsto V_f[s] = \ip{s}{f_{(\cdot)}}\in L^2(G),
\label{eq:CHS10}
\end{equation}
and the \emph{frame synthesis operator} is defined to be the adjoint $V_f^*$.

A general signal processing pipeline based on the continuous frame $f$ was defined in \cite{Ours1} as
\begin{equation}
\label{eq:end-to-end-pssp}
\cH\ni s \mapsto V_{\tilde{f}}^*T(r\circ V_f[s]).
\end{equation}
where $f$ and $\tilde{f}$ are canonical dual frames, $T: L^2(G)\rightarrow L^2(G)$ is a linear operator, and $r:\CC\rightarrow\CC$ is a pointwise nonlinearity. Signal processing tasks of this form are used in a multitude of applications, including multipliers
\cite{ex1,ex2,New_mult0,New_mult1,New_mult2}
(with applications, for example, in audio analysis
\cite{ex3}
\textcolor{black}{and increasing signal-to-noise ratio}
\cite{ex4}), signal denoising, e.g.,  wavelet shrinkage denoising
\cite{ex5,ex6}
and Shearlet denoising
\cite{ex7}, and phase vocoder \textcolor{black}{(see the classical papers \cite{phase_vocoder1,phase_vocoder2,vocoder_imp}, more modern approaches \cite{New_vocoder0,New_vocoder1,Ottosen2017APV,ltfatnote050}, and the book survey \cite{vocoder_book}).}

\paragraph{\textcolor{black}{Cubature vs. discrete frame discretizations.}}
The frame synthesis operator $ V_{\tilde{f}}^*$ is computed by integrating over $G$, with the formula
\begin{equation}
    V_{\tilde{f}}^*F = \int_G F(g)\tilde{f}_g dg.
    \label{eq:synth1}
\end{equation}
In \cite{Ours1} it was proposed to discretize the integration in (\ref{eq:synth1}) by a Monte Carlo cubature sum over $g$ \textcolor{black}{
\begin{equation}
    V_{\tilde{f}}^*F \approx \frac{\mu(G')}{K}\sum_{k=1}^K F(g_k)\tilde{f}_{g_k},
    \label{eq:synth_MC1}
\end{equation}
where $\{g_k\}_{k=1}^K$ are independent random samples from $G'\subset G$, where  $G'$ contains most of the energy of $F\in L^2(G)$, and has finite measure $\mu(G')<\infty$.} 
\textcolor{black}{Note that standard discretizations of a continuous frame deal with sampling from the continuous system a discrete frame, satisfying the frame inequality (see, e.g., \cite{Time_freq,Ten_lectures}), and not directly to approximate (\ref{eq:synth1}) with a cubature  (\ref{eq:synth_MC1}). }

\textcolor{black}{
The cubature approach to discretization has an advantage over continuous-to-discrete frame discretizations when working with highly redundant continuous frames. For example, consider a family of STFT frames $\{f_{g,c}\}_{g\in G}$ parametrized by $c$, where $g=(t,\w)\in G=\RR^2$ is the time-frequency parameter, and the window function $f_c$ is a Gaussian of time variance $c\in [a,b]$, where $0<a<b<\infty$. The unified family $\{f_{g,c}\}_{(g,c)\in G\times [a,b]}$ is a continuous frame that we call a \emph{redundant time-frequency system}. Suppose that we operate differently on coefficients corresponding to different window variances in our signal processing pipeline. To discretize such a pipeline `faithfully,' the sample points in $G\times [a,b]$ should be well spread and cover all axes roughly evenly. Now, if we discretize $\{f_{g,c}\}_{(g,c)\in G\times [a,b]}$  using the continuous-to-discrete frame approach, the discrete frame needs only to satisfy the frame inequality, and there is no requirement for well spread sample points. Indeed, even by fixing the parameter $c$ to one value $c_0$, we can sample $\{f_{g,c_0}\}_{g\in G}$ to a discrete frame. On the other hand, randomly sampling $G\times [a,b]$ will yield sample points which are well spread in all axes.
In this paper, we focus on the cubature approach to discretization  of continuous frames, using random samples.}

\paragraph{\textcolor{black}{Monte Carlo discretizations of continuous frames.} }
  The non-regular Monte Carlo discretization (\ref{eq:synth_MC1}) has the advantage that the number of samples required for some error tolerance does not depend on the dimension of $G$, but only on the volume $\mu(G')$. It is thus important to know how large the volume of $G'$ needs to be for some error tolerance. 
Another level of discretization required by (\ref{eq:end-to-end-pssp}) is using discrete signals $s$, namely, signals in a finite dimensional subspace $V_M\subset \cH$ of dimension/resolution $M$. For example, $\cH$ can be $L^2(\RR)$ and $V_M$ can be an $M$ dimensional spline space. 

The above two levels of discretization turn out to be closely related. In \cite{Ours1}, the class of \emph{linear volume discretizable frames} (LVD frames, see Definition \ref{D:linear area discretizable}) was introduced: frames for which the volume of $G'$ required for some error tolerance is proportional to the resolution $M$. 
Any continuous frame which is LVD allows an efficient Monte Carlo implementation, and thus checking the LVD property is very important.  

We prove in this paper that the CWT and the localizing time-frequency transform (LTFT) \cite{Ours1} are LVD.  
The LTFT  is a continuous frame for which $G$ is interpreted as the cross product of the time-frequency plane with a third axis, representing the uncertainty balance between time accuracy and frequency accuracy. This enhanced time-frequency plane has a richer set of time-frequency atoms than standard time-frequency transforms, and thus improves time-frequency signal processing methods like phase vocoder \cite{Ours1}. We prove in this paper that the complexity of the LTFT Monte Carlo method is \textcolor{black}{asymptotically} equivalent to the complexity of 2D time-frequency methods with FFT implementations, namely $O(Mlog(M))$.  

This work is a direct continuation of \cite{Ours1}, in which we developed the Monte Carlo signal processing in phase space theory, but did not prove important results for the motivating examples. In the current paper we prove that the LTFT is a continuous frame, derive a closed form formula for the frame operator of the LTFT, and prove that it is an LVD frame.

\textcolor{black}{We note that previous randomized methods for discretizing continuous frames, namely \emph{relevant sampling}  \cite{relevantS1,relevantS2,relevantS3,relevantS4}, operate in the continuous-to-discrete frame regime, and only require the sampled family to  satisfy the frame inequality. 
}

\section{Background: harmonic analysis in phase space}

In this section we review the theory of continuous frames and give the two important examples of the STFT and the CWT. 
 By convention, all Hilbert spaces in this paper are assumed to be separable. 
The Fourier transform $\cF$ is defined on signals $s\in L^1(\RR)$ by
\begin{equation}
[\cF s](\w) =\hat{s}(\w)= \int_{\RR}s(t)e^{-2\pi i \w t}dt, \quad [\cF^{-1}\hs](t) = \int_{\RR}\hs(\w)e^{2\pi i \w t}d\w, 
\label{eq:Fourier_trans}
\end{equation} 
and extended by density to $L^2(\RR)$ as usual.

\subsection{Continuous frames}

A continuous frame is a system of atoms/signals with general properties which guarantee that the decomposition of signals to the frame elements, and the recombination of the frame elements to signals, are both stable and allow perfect reconstruction. 
The following definitions and claims are from \cite{Cframe1} and \cite[Chapter 2.2]{Fuhr_wavelet}, with notation adapted from the latter.
The measure $\mu$ of a measurable space $G$ is called $\sigma$-\emph{finite} if there is a countable set of measurable sets $X_1,X_2,\ldots\subset G$ with $\mu(X_n)<\infty$ for each $n\in\NN$, such that $\bigcup_{n\in\NN}X_n=G$. A topological space $G$ is called locally compact if for every point $g\in G$ there exists an open set $U$ and a compact set $K$ such that $g\in U\subset K$.

\begin{definition}
\label{CSSframe}
Let $\cH$ be a Hilbert space, and $(G,\mathcal{B},\mu)$ a locally compact topological space with Borel sets $\mathcal{B}$, and $\sigma$-finite Borel measure $\mu$.
 Let $f:G\rightarrow \cH$ be a weakly measurable mapping, namely for every $s\in\cH$
\[g\mapsto \ip{s}{f_g}\]
is a measurable function $G\rightarrow\CC$.
For any $s\in\cH$, we define the \emph{coefficient function}
\begin{equation}
V_f[s]:G\rightarrow \CC \quad , \quad V_f[s](g)=\ip{s}{f_g}_{\cH}.
\label{eq:CSS2frame}
\end{equation}
\begin{enumerate}
	\item
	We call $f$ a \emph{continuous frame}, if $V_f[s]\in L^2(G)$ for every $s\in\cH$, and there exist constants $0<A\leq B<\infty$ such that
	\begin{equation}
	A\norm{s}_{\cH}^2 \leq \norm{V_f[s]}_2^2 \leq B\norm{s}_{\cH}^2
	\label{eq:FB}
	\end{equation}
	for every $s\in\cH$.
	\item
	We call $\cH$ the \emph{signal space},
	 $G$ \emph{phase space}, $V_f$ the \emph{analysis operator}, and $V_f^*$ the \emph{synthesis operator}.
	\item
	We call the frame $f$ \emph{bounded}, if there exist a constant $0<C\in\RR$ such that
\[\forall g\in G\ , \norm{f_g}_{\cH}\leq C.\]
\item
We call $S_f=V_f^*V_f$ the \emph{frame operator}. 
\item
	We call $f$ a \emph{Parseval} continuous frame, if $V_f$ is an isometry between $\cH$ and $L^2(G)$.
\end{enumerate}
\end{definition}

\begin{remark}
A Parseval frame is a continuous frame with frame bounds that can be chosen as $A=B=1$.
\end{remark}

Given a continuous frame, a concrete formula for the synthesis operator is given 
 by the weak integral  \cite[Theorem  2.6]{Cframe1}
\begin{equation}
V_f^*[F] = \int^{\rm w}_G F(g)f_g dg.
\label{eq:inver_proj}
\end{equation}
\textcolor{black}{
This integral is defined by
\begin{equation}
\ip{q}{\int^{\rm w}_G F(g)f_gdg} = \int_G \overline{F(g)}\ip{q}{f_g}dg,
\label{eq:a2}
\end{equation}
}
 where $\int^{\rm w}_G F(g)f_gdg$ denotes the vector corresponding to the continuous functional defined in the right-hand-side of (\ref{eq:a2}), whose existence is guaranteed by the Riesz representation theorem.  Such integrals are called \emph{weak vector integrals}, or Pettis integral \cite{Weak_Integral}.

\subsection{Generalized wavelet transforms}

An important class of Parseval continuous frames are \textcolor{black}{generalized} continuous wavelet transforms based on square integrable representations \cite[Chapters 2.3--2.5]{Fuhr_wavelet}. The general theory of wavelet transforms gives a procedure for constructing Parseval continuous frames, guaranteeing the properties of Definition \ref{CSSframe}. Moreover, \textcolor{black}{some} useful continuous frames that are not \textcolor{black}{structured as generalized wavelet transforms are built with} generalized wavelet transforms as building blocks. For example, the STFT is the restriction of the Schr\"odinger representation of the reduced Heisenberg group to the quotient group relative to the center \cite{Folland_Har}. Another example is the LTFT, which we construct as a combination of STFT and CWT wavelet atoms.
For more  on the theory of continuous wavelet transforms based on square integrable representations we refer the reader to \cite[Chapters 2.3--2.5]{Fuhr_wavelet}, and the classical papers \cite{gmp0,gmp}.
Next, we recall two well-known \textcolor{black}{generalized} continuous wavelet transforms.

\subsubsection{Transforms associated with time-frequency analysis}
\label{Transforms associated with time-frequency analysis}

We first recall \textcolor{black}{the basic operators} on which the STFT and CWT are based.
We formulate translation, modulation, and dilation, and give their formulas in the frequency domain.
\begin{definition}
Translation by $x$ of a signal $s:\RR\rightarrow\CC$ is defined by
\begin{equation}
[\mathcal{T}(x)s](t)= s(t-x).
\label{eq:trans0000}
\end{equation}
Modulation by $\w$ of a signal $s:\RR\rightarrow\CC$ is defined by
\begin{equation}
[\mathcal{M}(\w)s](t)= s(t)e^{2\pi i \w t}.
\label{eq:trans000}
\end{equation}
Dilation by $\tau$ of a signal $s:\RR\rightarrow\CC$ is defined by
\begin{equation}
[\mathcal{D}(\tau)s](t)= \tau^{-1/2}s(\tau^{-1} t).
\label{eq:trans00}
\end{equation}
\end{definition}
In (\ref{eq:trans00}), the dilation parameter $\tau^{-1}$ is interpreted as a frequency multiplier by $\tau$. Indeed, if $\hs$ is concentrated about frequency $z_0$, then $\cF[\mathcal{D}(\tau^{-1})s]$ is concentrated about frequency $\tau z_0$, as is shown in the following lemma. The proof of the following lemma is direct (see for example \cite[Sections 1.2 and 10]{Time_freq}).

\begin{lemma}
\label{Transform_lemma}
Translation, modulation, and dilation are unitary operators in $L^2(\RR)$ and take the following form in the frequency domain.
\begin{enumerate}
\item
$\cF(\mathcal{T}(x)s) =\mathcal{M}(-x)\hs$.
\item
$\cF( \mathcal{M}(\w)s)= \mathcal{T}(\w)\hs$.
	\item 
	$\cF (\mathcal{D}(\tau)s)=\mathcal{D}(\tau^{-1})\hs$.
\end{enumerate}
\end{lemma}

\subsubsection{The short time Fourier transform}

The following construction is taken from \cite{Time_freq,Folland_Har}.
Consider the signal space $L^2(\RR)$, and the time-frequency phase space $G=\RR^2$ with the usual Lebesgue measure $dxd\w$, where $x$ is called \emph{time} and $\w$ is called \emph{frequency}. Consider a function $f\in L^2(\RR)$ that we call the \emph{window function}. The STFT system is defined as
\[\{f_{x,\w}=\mathcal{T}(x)\mathcal{M}(\w)f\}_{(x,\w)\in G}.\]
The STFT system is a continuous Parseval frame for every $f\in L^2(\RR)$ satisfying $\norm{f}_2=1$.

\subsubsection{The 1D continuous wavelet transform}
\label{The 1D continuous wavelet transform}

The following construction is taken from \cite{Cont_wavelet_original,Ten_lectures}.
Consider the signal space $L^2(\RR)$, and the time-scale phase space $G=\RR \times (\RR\setminus\{0\})$ with the weighted Lebesgue measure $\frac{1}{\tau^2}d\tau dx$, where $x\in\RR$ is called \emph{time} and $\tau\in \RR\setminus\{0\}$ is called \emph{scale}. Consider a function $f\in L^2(\RR)$ satisfying the \emph{admissibility condition}
\[\int_{\RR}\frac{1}{z}\abs{\hat{f}(z)}^2dz=A_f<\infty\]
that we call the \emph{mother wavelet}. The CWT system is defined to be
\[\{f_{x,\tau}=\mathcal{T}(x)\mathcal{D}(\tau)f\}_{(x,\tau)\in G}.\]
The CWT system is a Parseval continuous frame if $A_f=1$.

Next, we show how the CWT atoms are interpreted as time-frequency atoms, and the CWT is interpreted as a time-frequency transform.
Here, by changing variable $\w=\frac{1}{\tau}$, we obtain the Parseval frame
\[\{\mathcal{T}(x)\mathcal{D}(\w^{-1})f\}_{(x,\w)\in G'}\]
with $G'=\RR\times (\RR\setminus\{0\})$ with the standard Lebesgue measure $dxd\w$.
The parameter $\w$ is interpreted as frequency.   

\section{Phase space signal processing and its stochastic approximation}
\label{Phase space signal processing and its stochastic approximation}

In this section we summarize the theory of stochastic phase space signal processing introduced in \cite{Ours1}.  We moreover offer a motivation for the analysis and synthesis formulations of the signal processing pipelines. Along with stochastic phase vocoder, which was introduced in \cite{Ours1}, we formulate two additional examples of stochastic phase space signal processing, namely randomized phase space multipliers and shrinkage.

\subsection{Non-Parseval phase space signal processing}
\label{Non-Parseval phase space signal processing}

In the case of a Parseval frame, a signal processing method in phase space is any procedure that maps a signal $s\in\cH$ to phase space, applies a pointwise nonlinearity $r:\CC\rightarrow\CC$ on $V_f[s]$ (by $g\mapsto r\big(V_f[s](g)\big)$), applies a linear operator $T$, and synthesizes back to a signal. Namely, we consider procedures of the form
\textcolor{black}{
\[s \mapsto V_f^*T (r\circ V_f [s]).\]}
In this subsection we show the extension of this procedure to non-Parseval phase space signal processing based on continuous frames.

In case $f$ is a continuous frame, we consider two approaches to phase space signal processing, both generalizing the Parseval frame formulation. 
\textcolor{black}{A basic guideline of the construction is to derive  formulas which can be discretized efficiently.} Equations (\ref{eq:CSS2frame}) and (\ref{eq:inver_proj}) give constructive formulas for $V_f$ and $V_f^*$, and we show in this paper that the Monte Carlo formulations efficiently approximate these formulas in the CWT, STFT, and LTFT transforms.  
   \textcolor{black}{We moreover only treat continuous frames for which $S_f^{-1}$ can be discretized efficiently in the signal space. For example, this is the case for Parseval frames, since $S_f=I$, and for the LTFT (Definition \ref{The localizing time-frequency continuous frame}), as we show in Subsection \ref{The frame operator of the LTFT}. 
  Hence, we allow in our general formulations of phase space signal processing the application of $V_f,V_f^*$, and $S_f^{-1}$, in addition to the phase space operator $T$ and the pointwise non-linearity $r$.}  

\textcolor{black}{In the following  we present the synthesis and analysis formulations of phase space signal processing, where we work with the frame $f=\{f_g\}_{g\in G}$ in the synthesis step, and with the canonical dual frame \cite{Cframe1}  $S_f^{-1}f:=\{S_f^{-1}f_g\}_{g\in G}$  in the analysis step, or vice versa.} 

\paragraph{Synthesis phase space signal processing.}

Synthesis phase space signal processing is based on the synthesis operator $V_f^*$ as the basic transform.  In this approach, we view phase space signal processing as the procedure of decomposing signals to their different atoms, modifying the coefficients of the atoms, and synthesizing these modified coefficients.
The synthesis transform $V_f^*$, which combines atoms to signals, serves as the inverse transform of the pipeline. \textcolor{black}{For the pipeline to become $I$ when $T$ and $r$ are trivial, we choose the forward transform as $V_f^{+*}$ -- the pseudo inverse of $V_f^*$ (see Lemma \ref{Lemm_PI}.\ref{Lemm_PI:2} in Appendix \ref{Pseudo inverse of teh analsis operator}).}
Hence, synthesis phase space signal processing is the pipeline
\begin{equation}
s\mapsto V_f^* T \big(r\circ V_f^{+*}[s]\big).
\label{eq:PSSPS}
\end{equation}
\textcolor{black}{Since we assume that $S_f^{-1}$ can be discretized efficiently, by Lemma \ref{Lemm_PI}.\ref{Lemm_PI:0}} of Appendix \ref{Pseudo inverse of teh analsis operator}, we implement (\ref{eq:PSSPS}) by
\begin{equation}
s\mapsto V_f^* T \big(r\circ V_fS_f^{-1}s\big).
\label{eq:PSSPS22}
\end{equation}


\paragraph{Analysis phase space signal processing.}

Analysis phase space signal processing takes the analysis operator $V_f$ as the basic transform. Here, we view phase space signal processing as the procedure of computing the correlation of the signal with the different atoms to obtain coefficients, modifying these coefficients, and outputting the signal that has coefficients closest to these modified coefficients. 
 \textcolor{black}{Given a function $F\in L^2(G)$, the pseudo inverse of analysis $V_f^{+}F$  gives the signal that best fits $F$ in the sense $V_f^{+}F = {\rm arg}\min_{s\in\cH} \norm{V_fs-F}_2$ (see Lemma \ref{Lemm_PI}.\ref{Lemm_PI:20} in Appendix \ref{Pseudo inverse of teh analsis operator}). Hence, $V_f^{+}$ is used as the inverse transform in the pipeline.}
Analysis phase space signal processing is the pipeline
\begin{equation}
s\mapsto V_f^{+} T r\circ \big(V_f[s]\big).
\label{eq:PSSPA}
\end{equation}
\textcolor{black}{Since  $S_f^{-1}$ has an efficient discretization, by Lemma \ref{Lemm_PI}.\ref{Lemm_PI:1} of Appendix \ref{Pseudo inverse of teh analsis operator} we have $V_f^{+}=S_f^{-1}V_f^*$}, so we implement (\ref{eq:PSSPA}) by
\begin{equation}
s\mapsto S_f^{-1}V_f^* T r\circ \big(V_f[s]\big).
\label{eq:PSSPA2}
\end{equation}

\subsection{Phase space operators}

\textcolor{black}{One example of the pipeline (\ref{eq:end-to-end-pssp}) is when $T$ is an integral operator in $L^2(G)$.
More generally, it is common to define classes of operators by assuming boundedness between some choice of a domain and a range of the operators. Kernel theorems then prove that such operators can be written as integral operators, or more generally defined by the application of a kernel, which generalizes the classical notion of a matrix operator. A classic example is the Schwartz kernel theorem \cite[Section 5.2]{Kernel0}. In the context of continuous frames, Feichtinger’s kernel theorem is for the STFT \cite{Kernel1}\cite[Theorem 14.4.1]{Time_freq}, and for general wavelet transforms see \cite{Kernel2}. In this paper, we take the route of \cite{Ours1}, and define integral operators $T$ directly.}

\textcolor{black}{
\begin{definition}[\cite{Ours1}]
\label{phase_space_operator}
Let $T$ be a bounded linear operator on $L^2(G)$, where $G$ is a locally compact topological space with $\sigma$-finite Borel measures.
\begin{enumerate}
	\item
	We call $T$ a \emph{phase space integral operator (PSI operator)} if there exists a measurable function $R:G\times G\rightarrow\CC$ 
	with $R(\cdot,g)\in L^2(G)$ for almost every $g\in G$, such that for every $F\in L^2(G)$
\begin{equation}
TF = \int_{G} R(\cdot,g)F(g)dg.
\label{eq:PSO}
\end{equation}
\item
 A phase space integral operator $T$ is called \emph{uniformly square integrable}, if there is a constant $D>0$ such that for almost every $g\in G$
 \begin{equation}
 \label{eq:USI}
  \norm{R(\cdot,g)}_{L^2(G)}=\sqrt{\int_{G}  \abs{R(g',g)}^2 dg'} \leq D.   
 \end{equation}
\end{enumerate}
\end{definition}
}


\subsection{Linear volume discretizable frames}
\label{Monte-Carlo sample sets in phase space}

When constructing the Monte Carlo approximation \textcolor{black}{(\ref{eq:synth_MC1})} to the synthesis operator  (\ref{eq:inver_proj}), we need to decide how to \textcolor{black}{randomly sample the points $\{g_n\}\subset G$.} For that, we need to restrict the variable $g$ to a subset $G'\subset G$ of finite measure, so $G'$ can be normalized to a probability space. 
To restrict a function $F\in L^2(G)$ to the subset $G'\subset G$ we multiply $F$ by the indicator function $\mathbf{1}_{G'}$ of the set $G'$. Note that $\norm{\mathbf{1}_{G'}}_{L^1(G)}=\mu(G')$. Hence, in a more general analysis, we restrict functions $F\in L^2(G)$ by multiplying them with a function $\psi\in L^1(G)$, which we call an \emph{envelope}. The Monte Carlo sample points $g$ are now drawn from $G$ with the probability density $\frac{\psi(g)}{\norm{\psi}_1}$.
The choice of $G'$, or more generally $\norm{\psi}_1$, is closely linked to the resolution of the signal space, as we explain next.

 In application, signals are given as discrete entities with finite \textcolor{black}{resolution/ dimension} $M$. We thus define discrete signals of resolution $M$ as vectors in an $M$ dimensional subspace $V_M\subset\cH$, where for a fixed  $M$, $V_M$ is seen as the space of signals of resolution $M$. 
 Moreover, we sometimes restrict the space of signals $\cH$ to a nonlinear subset $\mathcal{R}\subset\cH$ with desired properties, e.g., smoothness assumptions.
 To accommodate an asymptotic analysis, we call the sequence of subspaces $\{V_M\}_M$, where $V_M$ is of dimension $M$ for each $M$, a \emph{discretization} of $\mathcal{R}$, if for every $s\in \mathcal{R}$ and every $M\in\NN$ there exists $s_M\in V_M$ such that
\[\lim_{M\rightarrow\infty}\norm{s_M-s}_{\cH}=0.\]
We also allow nonlinear spaces $V_M$ having $M$ dimensional linear closures. Such spaces can still be seen as having resolution $M$, since each signal in $V_M$ can be represented using $M$ scalars. 
 
 Apart from describing real-life signals, finite resolution plays to our advantage when constructing $G'$. Namely, there is a natural definition that relates the volume $\mu(G')$ to $M$.

 \begin{definition}[Linear volume discretization \cite{Ours1}]
\label{D:linear area discretizable}
Let $f:G\rightarrow\cH$ be a continuous frame,
and let $\{V_M\subset \cH\}_{m=1}^{\infty}$ be a discretization, with \textcolor{black}{finite} dimension $M$ for each $M\in\NN$. 
\begin{enumerate}
	\item 
	The continuous frame $f$ is called \emph{linear volume discretizable} (LVD) with respect to the  discretization $\{V_M\}_{M=1}^{\infty}$, if for every error tolerance $\e>0$ there is a constant $C^{\e}>0$ and $M_0\in\NN$, such that for any $M\geq M_0$ there is an envelope $\psi_M$ with 
\begin{equation}
\norm{\psi_M}_1 \leq C^{\e}M
\label{eq:lin_area1}
\end{equation}
such that for any $s_M\in  V_M$,
\begin{equation}
\frac{\norm{V_f[s_M] - \psi_M V_f[s_M]}_2}{\norm{V_f[s_M]}_2} < \e.
\label{eq:lin_area2}
\end{equation}
\item
For a linear volume discretizable continuous frame $f$ with respect to  $\{V_M\}_{M=1}^{\infty}$, and a fixed tolerance $\e>0$ with a corresponding fixed $C^{\e}$ and envelope sequence $\{\psi_M\}_{M=1}^{\infty}$ satisfying (\ref{eq:lin_area1}) and (\ref{eq:lin_area2}), we call $f$ together with $\{V_M\}_{m=1}^{\infty}$ and $\{\psi_M\}_{M=1}^{\infty}$, an \emph{$\e$-linear volume discretization} of $f$.
\end{enumerate}
\end{definition}

\textcolor{black}{In the LVD definition, (\ref{eq:lin_area1}) and (\ref{eq:lin_area2}) tell us that it is enough to work with a domain of volume $O(M)$ in phase space when analyzing discrete signals of dimension $M$ up to small error. The constant $C^{\e}$ of (\ref{eq:lin_area1}) may increase as the error tolerance $\e$ of (\ref{eq:lin_area2})  decreases.}


\subsection{Error in discrete stochastic phase space signal processing}
\label{Error in discrete stochastic phase space signal processing}

In this subsection we recall the error bound from \cite{Ours1} of discrete stochastic phase space signal processing. 
\textcolor{black}{For discrete signals of resolution $M$, we can sample the $K$ random points in (\ref{eq:synth_MC1}) from a domain of volume $M$ in phase space using the LVD property. As a result, the error in the Monte Carlo approximation (\ref{eq:synth_MC1}) of the synthesis operator (\ref{eq:synth1}) is of order $O(\frac{\sqrt{M}}{\sqrt{K}})$. In this section we formulate this approximation analysis, and extend it to end-to-end signal processing pipelines of the form (\ref{eq:CHS10}). The following construction summarizes the setting from \cite{Ours1}.}

\begin{assumption}[Signal processing in phase space setting \cite{Ours1}]
\label{As2}
Let $f$ be a bounded continuous frame. Let $r:\CC\rightarrow\CC$ satisfy 
\begin{equation}
\abs{r(x)}\leq E\abs{x} 
\label{eq:Discrete_SPSO0}
\end{equation}
for some $E\geq 0$. 
Suppose that $f$ together with the discretization $\{V_M\}_{m=1}^{\infty}$ (of dimension ${\rm dim}(V_M)=M$),  and the envelopes $\{\psi_M\}_{M=1}^{\infty}$, is an $\e$-LVD, with constant $C^{\e}$.
Let $T:L^2(G)\rightarrow L^2(G)$ be a bounded operator that maps most of the energy in the support of $\psi_M$ to the support of the envelope $\eta_M$, in the sense
\begin{equation}
\norm{T\psi_M - \eta_M T \psi_M}_2 \leq \e.
\label{eq:Tmap0b}
\end{equation}
Here, the envelopes $\{\eta_M\}_{M=1}^{\infty}$ satisfy
\begin{equation}
\norm{\eta_M}_1 \leq C^{\e}M.
\label{eq:lin_area12b}
\end{equation}
%
%
%
Consider the signal processing pipeline $V_{M} \ni s\mapsto\mathcal{T}s$, where 
\begin{equation}
\mathcal{T}s = V_f^* Tr\circ V_f[S_f^{-1}s] \quad {\rm or} \quad  \mathcal{T}s = S_f^{-1}V_f^* T r\circ V_f[S_f^{-1}s].
\label{eq:pipe14b}
\end{equation}
 \textcolor{black}{Let $\{g^k\}_{k=1}^K$ be independent random samples from the distribution $\frac{\eta_M(g)}{\norm{\eta_M}_1}$, and $\{y^l\}_{j=l}^L$ independent random samples from the distribution $\frac{\psi_M(g)}{\norm{\psi_M}_1}$.}
Suppose that one of the following two discretizations of (\ref{eq:pipe14b}) is used. 
\begin{enumerate}
	\item \emph{Output stochastic signal processing}:
Consider the two stochastic approximations of the two pipelines (\ref{eq:pipe14b})
\begin{equation}
    \label{DMCSP1}
    \begin{split}
    & [\mathcal{T}s]^{\eta, K}=\frac{\norm{\eta_M}_1}{K}\sum_{k=1}^K \Big(T r\circ V_f[S_f^{-1}s]\Big)(g^k)f_{g^k} \\
    & \quad {\rm or} \quad \frac{\norm{\eta_M}_1}{K}S_f^{-1}\sum_{k=1}^K \Big(T r\circ V_f[s]\Big)(g^k)f_{g^k}
    \end{split}
\end{equation}
respectively.
Suppose that the number of Monte Carlo samples is $K=ZC^{\e}M$, where $Z>0$, and denote $[\mathcal{T}s]^{K}=[\mathcal{T}s]^{\eta, K}$.
	\item
	\label{item:SMC}\emph{Input-output stochastic signal processing}:
	Suppose that $T$ is a \textcolor{black}{uniformly square integrable} PSI operator.
Consider the two stochastic approximations of the two pipelines (\ref{eq:pipe14b})
\begin{equation}
    \label{DMCSP2}
    \begin{split}
     & [\mathcal{T}s]^{\eta, K; \psi, L}=\frac{\norm{\eta}_1\norm{\psi}_1}{KL}\sum_{j=1}^L\sum_{k=1}^K R(g^k,y^l)r\Big(V_f[S_f^{-1}s](y^l)\Big)f_{g^k} \\
      & \quad {\rm or} \quad S_f^{-1}\frac{\norm{\eta}_1\norm{\psi}_1}{KL}\sum_{j=1}^L\sum_{k=1}^K R(g^k,y^l)r\Big(V_f[s](y^l)\Big)f_{g^k}
    \end{split}
\end{equation}
 respectively.
Suppose that the number of Monte Carlo samples is $K=L=ZC^{\e}M$, where $Z>0$, and denote $[\mathcal{T}s]^{ K}=[\mathcal{T}s]^{\eta, K; \psi, K}$.
\end{enumerate}
\end{assumption}

\begin{theorem}[\cite{Ours1}]
\label{Main_Theorem12}
Consider the setting of Assumption \ref{As2}.
Then, for every $M\in\NN$ and every $s\in V_M$,
	\begin{equation}
	\frac{\mathbb{E}\Big(\norm{\mathcal{T}s -[\mathcal{T}s]^{K}}_{\cH}^2\Big)}{{\norm{s}_{\cH}^2} } = O(Z^{-1})+O(\e^2),
	\label{eq:Discrete_SP1}
	\end{equation}
	\textcolor{black}{where the expected value is with respect to the random samples $\{g^k\}_{k=1}^K$ and $\{y^l\}_{l=1}^L$.}
\end{theorem}

Theorem \ref{Main_Theorem12} bounds the expected error of the Monte Carlo approximation. In \cite{Ours1}, an additional version of the error bound was presented, where the error is shown to be $O(Z^{-1})+O(\e^2)$ in high probability. \textcolor{black}{As evident by Theorem \ref{Main_Theorem12}, the signal processing pipelines (\ref{eq:pipe14b}) can be approximated by the Monte Carlo methods (\ref{DMCSP1}) and (\ref{DMCSP2}), using $O(M)$ samples, if the frame is LVD. It is hence important to show the LVD property of frames, before they are used in a stochastic signal processing pipeline. In Section \ref{Discretization of phase space in time frequency analysis} we show that the STFT, CWT, and LTFT are LVD.}

\subsection{Examples}

We recall in this subsection the diffeomorphism example of \cite{Ours1}, and two additional examples of stochastic signal processing in phase space. We use the short notation $C$ for the normalization constant of either of the formulas in (\ref{DMCSP1}) and (\ref{DMCSP2}).

\subsubsection{Stochastic phase space diffeomorphism}
\label{Stochastic phase space diffeomorphism}
Let $f:G\rightarrow\cH$ be a bounded continuous frame, with bound $\norm{f_g}_{\cH}\leq C$, based on a Riemannian manifold $G$.
Let $d:G\rightarrow G$ be a diffeomorphism (invertible smooth mapping with smooth inverse), with Jacobian $J_d\in L^{\infty}(G)$. 
Consider the diffeomorphism operator $T$, defined for any $F\in L^2(G)$ by
\[[TF](g) = F\big(d^{-1}(g)\big).\]
Let $r:\CC\rightarrow\CC$ be a function that preserve modulus. Namely, $\abs{r(z)} = \abs{z}$ for every $z\in\CC$. 

By sampling the points $d(g_n)$, \emph{synthesis Monte Carlo phase space diffeomorphism}, \textcolor{black}{based on (\ref{eq:PSSPS22})}, takes the form
\begin{equation}
s\mapsto   C \sum_n r\big(V_f [S_f^{-1}s](g_n)\big)f_{d(g_n)},
\label{eq:PSSPS22MC0100}
\end{equation}
and analysis Monte Carlo phase space diffeomorphism, \textcolor{black}{based on (\ref{eq:PSSPA2}),} takes the form
\begin{equation}
s\mapsto C S_f^{-1} \sum_n r\big(V_f[s](g_n)\big)f_{d(g_n)}.
\label{eq:PSSPA2MC0100}
\end{equation}

\begin{example}[Integer time stretching phase vocoder \textcolor{black}{\cite[Section 7.4.3]{vocoder_book}}]
\label{Phase vocoder example}
A time stretching phase vocoder is an audio effect that slows down an audio signal without dilating its frequency content. In the classical definition, $G$ is the time frequency plane, and $V_f$ is the STFT. \textcolor{black}{The} phase vocoder can be formulated as phase space signal processing in case the signal is dilated by an integer $\Delta$. 
 For an integer $\Delta$,
we consider the diffeomorphism operator $T$ with
\[d(g_1,g_2)=(\Delta g_1,g_2),\]
and consider the nonlinearity $r$, defined by $r(e^{i\theta}a)=e^{i\Delta\theta}a$, for \textcolor{black}{$a\in\RR_+$ and $\theta\in\RR$}. The phase vocoder is defined to be
\[s\mapsto V_f^* T r\circ V_f[s].\]
Note that the STFT is a Parseval frame, \textcolor{black}{so synthesis (\ref{eq:PSSPS22}) and analysis (\ref{eq:PSSPA2}) signal processing are identical}.
\end{example}

\subsubsection{Phase space multipliers}

Given a function $h\in L^{\infty}(G)$ called the \emph{symbol}, and a continuous frame $\{f_g\}_{g\in G}$,  the linear operator $H$ in $\cH$, defined for every $s\in\cH$ by
\[Hs = V_f^*hV_f[s],\]
\textcolor{black}{is called a \emph{Continuous frame multiplier}  \cite{New_mult2,New_mult0,New_mult1}.}
Here, $hV_f[s]$ is the function in $L^2(G)$ that assigns the value $h(g)V_f[s](g)$ to each $g\in G$.
The synthesis and analysis \emph{phase space multiplier} based on the symbol $h$ and the frame $f$ are defined respectively by
\[s\mapsto V_f^*hV_f[S_f^{-1}s] , \quad s\mapsto S_f^{-1}V_f^*hV_f[s]. \]
Now, synthesis and analysis Monte Carlo phase space multipliers take the following forms respectively
\[
s\mapsto   C \sum_n h(g_n)V_f [S_f^{-1}s](g_n)f_{g_n}, \quad s\mapsto   C S_f^{-1}\sum_n h(g_n)V_f [s](g_n)f_{g_n},
\]
for the appropriate normalization $C$.

\subsubsection{Phase space shrinkage}

 \textcolor{black}{One method of signal denoising is \emph{shrinkage}, e.g.,  wavelet shrinkage denoising 
\cite{ex5,ex6},
and Shearlet denoising \cite{ex7}.}
Let $r:\CC\rightarrow\CC$ be a \emph{denoising operator}, e.g., the soft thresholding function with threshold $\l$
\[r(x) = e^{i{\rm Arg}(x)}\max\{0,\abs{x}-\l\},\]
where ${\rm Arg}(x)$ is the argument of the complex number $x$, namely $x=e^{i{\rm Arg}(x)}\abs{x}$. More generally, a denoising operator is any function $r(x)$ that decreases small values of $x$ and approximately retains large values of $x$.
The synthesis and analysis phase space shrinkage based on the denoising operator $r$ and the frame $f$ are defined respectively by
\[s\mapsto V_f^*r\circ V_f[S_f^{-1}s] , \quad s\mapsto S_f^{-1}V_f^*r\circ V_f[s]. \]
Now, synthesis and analysis Monte Carlo phase space shrinkage take the following forms respectively
\[
s\mapsto   C \sum_n r\Big(V_f [S_f^{-1}s](g_n)\Big)f_{g_n}, \quad s\mapsto   C S_f^{-1}\sum_n r\Big(V_f [s](g_n)\Big)f_{g_n},
\]
for the appropriate normalization $C$.

%
%

\section{Analysis of the localizing time-frequency transform}
\label{The localizing time-frequency transform}

In \cite{Ours1} the LTFT transform was introduced, without proving that it is a bounded continuous frame, without giving a formula for the frame operator, and without proving LVD. In this section we recall the definition of the LTFT, prove that it is a continuous frame, and give a formula for the frame operator. In Subsection \ref{Linear volume discretization of the LTFT} we prove the the LTFT is a LVD frame.

\subsection{The localizing time-frequency transform}

 \textcolor{black}{The LTFT is a highly redundant time-frequency transform that was derived in \cite{Ours1} from classical time-frequency transforms in two steps. First, combining the STFT with the CWT, the LTFT represent low and high frequencies by STFT atoms, and middle frequencies by CWT atoms. Second, the time-frequency plane is enhanced by adding a third dimension, that controls the number of oscillations in the atoms. For the signal processing motivation behind the construction, see \cite[Section 6.2.3]{Ours1}}
 
 \textcolor{black}{
Before recalling the definition of the LTFT, we formalize geometric properties of time-frequency atoms.
\begin{definition}
Let $q\in L^2(\RR)$. 
\begin{itemize}
    \item 
    The \emph{time-expected value} and the \emph{frequency-expected value} of $q$ are defined respectively as
\[e^{\rm T}_q = \int_{\RR}t\abs{q(t)}^2dt , \quad e^{\rm F}_q = \int_{\RR}\w\abs{\hat{q}(\w)}^2d\w,\]
whenever these integrals are finite.
 The function $q$ is said to be \emph{centered about $x$ in time} if $e^{\rm T}_q=x$, and \emph{centered about $\w$ in frequency} if $e^{\rm F}_q=\w$.
 \item
 If $q$ is supported on the interval $(t_1,t_2)$ and centered about $\kappa$ in  frequency,  the \emph{number of oscillations} in $q$ is defined to be $\kappa(t_2-t_1)$.
\end{itemize}
\end{definition}}
 
\begin{definition}[The localizing time-frequency continuous frame \cite{Ours1}]
\label{The localizing time-frequency continuous frame}
\textcolor{black}{Let $f$ be a non-negative real valued window supported on $(-1/2,1/2)$. Let $0<\tau_1<\tau_2\in\RR$, and let $\mu_{\tau}$ be a weighted Lebesgue measure on $[\tau_1,\tau_2]$ with $\mu_{\tau}([\tau_1,\tau_2])=1$. For each $\tau\in[\tau_1,\tau_2]$ define the \emph{low and high transition frequencies} as two scalars $0<a_{\tau}<b_{\tau}\in\RR$.}
 The LTFT atoms are defined for $(x,\w,\tau)\in \RR^2\times [\tau_1,\tau_2]$, where $x$ represents time, $\w$ frequencies, and $\tau$  the number of wavelet oscillations, by
\begin{equation}
\begin{split}
f_{x,\w,\tau}(t)  
 & = \left\{
\begin{array}{ccc}
	\mathcal{T}(x) \mathcal{M}(\w) \mathcal{D}(\tau/a_{\tau}) f(t) & {\rm if} & \abs{\w}<a_{\tau} \\
	 \mathcal{T}(x)\mathcal{M}(\w)  \mathcal{D}(\tau/\w) f(t) & {\rm if} & a_{\tau}\leq\abs{\w}\leq b_{\tau} \\
	\mathcal{T}(x) \mathcal{M}(\w) \mathcal{D}(\tau/b_{\tau})f(t)  & {\rm if} &  b_{\tau}<\abs{\w}
\end{array}
\right. \\
& =
\left\{
\begin{array}{ccc}
	\sqrt{\frac{a_{\tau}}{\tau}}f\big(\frac{a_{\tau}}{\tau}(t-x)\big)e^{2\pi i \w (t-x)} & {\rm if} & \abs{\w}<a_{\tau} \\
	\sqrt{\frac{\w}{\tau}}f\big(\frac{\w}{\tau}(t-x)\big)e^{2\pi i \w (t-x)} & {\rm if} & a_{\tau}\leq\abs{\w}\leq b_{\tau} \\
	\sqrt{\frac{b}{\tau}}f\big(\frac{b_{\tau}}{\tau}(t-x)\big)e^{2\pi i \w (t-x)}  & {\rm if} &  b_{\tau}<\abs{\w}
\end{array}
\right.
\end{split}
\label{eq:LTFT_atom}
\end{equation}
\end{definition}

For $a_{\tau}\leq\abs{\w}\leq b_{\tau}$, we call the atoms $f_{x,\w,\tau}(t)=\sqrt{\frac{\w}{\tau}}f\big(\frac{\w}{\tau}(t-x)\big)e^{2\pi i \w (t-x)}$ of (\ref{eq:LTFT_atom}) CWT atoms. We adopt this terminology since $f_{x,\w,\tau}$ is the translation by $x$ and dilation by $\w^{-1}$ of the ``mother wavelet'' 
\begin{equation}
f_{\tau}(t)=\tau^{-1/2}f(t/\tau)e^{2\pi i t}.
\label{eq:false_mother}
\end{equation}
 However, the function $f_{\tau}$ is not really a mother wavelet, since for a.e. $\tau$ it does not satisfy the wavelet admissibility condition. 
Indeed, by the fact that $f$ is compactly supported, $\hat{f}$ is non-zero almost everywhere, so $\hat{f}_{\tau}(0)\neq 0$ for almost every $\tau$, and 
\begin{equation}
\int \frac{1}{z}\abs{\hat{f}_{\tau}(z)}^2 dz=\infty.
\label{eq:adDiv}
\end{equation}
The divergence of (\ref{eq:adDiv}) is not a problem in our theory, since the admissibility condition does not show up in the analysis of the LTFT.
Indeed, high frequencies are analyzed using STFT atoms and not CWT atoms.

\begin{theorem}
\label{LTFT_frame}
The LTFT system $\{f_{x,\w,\tau}\}_{(x,\w,\tau)\in \RR^2\times [\tau_1,\tau_2]}$ is a continuous frame \textcolor{black}{in $L^2(\RR)$}.
\end{theorem}

The proof is in Appendix \ref{proofs of g}.

Note that for each $\tau_1\leq \tau \leq \tau_2$, the support of the low and high frequency STFT windows are $2\tau/a_{\tau}$ and $2\tau/b_{\tau}$ respectively.
Let us consider three cases for $a_{\tau},b_{\tau}$. First, we may choose $a_{\tau}=a$ and $b_{\tau}=b$ constants. Second, if we want the supports of the low and high frequency STFT atoms to be the constants $J_1>J_2$ respectively, we choose $a_{\tau}=2\tau/J_1$ and $b_{\tau}=2\tau/J_2$. In this case, the LTFT can be viewed as a CWT transform with variable number of oscillations $\tau$ of the wavelet atoms, where any atom supported on an interval longer than $J_1$ or shorter than $J_2$ is truncated/extended to a STFT atom supported on an interval of length $J_1$ or $J_2$ respectively. 
Last, we can choose $\tau_1=\tau_2$, and obtain a hybrid STFT-CWT time-frequency transform with a 2-dimensional phase space.

 \textcolor{black}{Related to the LTFT construction, continuous warped time-frequency representations \cite{Wrapped0} tile the frequency axis arbitrarily, with the wavelet representation as a special case, and their discrete counterparts were proposed in the context of phase vocoder in \cite{Wrapped_Vocoder}.} \textcolor{black}{Moreover, the combination of the STFT with the CWT was studied in the past. Such frameworks, when based on group representations, are usually called affine Weyl-Heisenberg transforms (see e.g. \cite{WHA0,WHA1,WHA2}).
As opposed to this approach, the LTFT  combines the STFT with the CWT to a continuous frame, but not to a generalized wavelet transform (Parseval frame based on square integrable representation). Omitting the generalized wavelet and Parseval restrictions from the LTFT frame makes it more flexible and applicable to signal processing.  }

\subsection{The frame operator of the LTFT}
\label{The frame operator of the LTFT}

To accommodate a computationally efficient signal processing pipeline, we derive an explicit formula for $S_f^{-1}$.
Denote $\mathcal{J}_{\tau}^{\rm low}=\{\w\ |\ \abs{\w}<a_{\tau}\}$, $\mathcal{J}_{\tau}^{\rm mid}=\{\w\ |\ a_{\tau}\leq\abs{\w}\leq b_{\tau}\}$, $\mathcal{J}_{\tau}^{\rm high}=\{\w\ |\  b_{\tau}<\abs{\w}\}$ . 
Let ${\rm band}\in \{{\rm low},{\rm mid},{\rm high}\}$, 
and denote
\begin{equation}
\hf^{\rm band}(\tau,\w; z-\w) = \mathbf{1}^{\rm band}_{\tau}(\w)\left\{
\begin{array}{ccc}
	\sqrt{\frac{\tau}{a_{\tau}}}\hf\big(\frac{\tau}{a_{\tau}}(z-\w)\big) & {\rm if} & \abs{\w}<a_{\tau} \\
	\sqrt{\frac{\tau}{\w}}\hf\big(\frac{\tau}{\w}(z-\w)\big) & {\rm if} & a_{\tau}\leq\abs{\w}\leq b_{\tau} \\
	 \sqrt{\frac{\tau}{b_{\tau}}}\hf\big(\frac{\tau}{b_{\tau}}(z-\w)\big)  & {\rm if} & b_{\tau}\leq\abs{\w}.
\end{array}
\right.
\label{eq:LTFT_atom_freq_hh}
\end{equation}
where $\mathbf{1}^{\rm band}_{\tau}$ is the characteristic function of $\mathcal{J}_{\tau}^{\rm band}$.

\begin{definition}
\label{sub-frame filters}
The \textcolor{black}{\emph{sub-frame filter kernels}} $\hat{S}_f^{\rm low}$, $\hat{S}_f^{\rm mid}$ and $\hat{S}_f^{\rm high}$ are the functions $\RR\rightarrow\CC$ defined by
\begin{equation}
\hat{S}_f^{\rm band}(z) =  \int_{{\tau}_1}^{{\tau}_2}  \int_{\mathcal{J}^{\rm band}_{\tau}}     \abs{\hf^{\rm band}(\tau,\w; z-\w)}^2  d\w   d{\tau}.
\label{eq:Sf_hat}
\end{equation}
The \textcolor{black}{\emph{frame filter kernel}} $\hat{S}_f:\RR\rightarrow\CC$ is defined by
\[\hat{S}_f=\hat{S}_f^{\rm low}+\hat{S}_f^{\rm mid}+\hat{S}_f^{\rm high}.\]
\end{definition}

\begin{proposition}
\label{Prop_SfLTFT}
For any  $s\in L^2(\RR)$,
the frame operator is given in the frequency domain by
\[\cF [S_f s] = \hat{S}_f^{\rm low} \hat{s} +\hat{S}_f^{\rm mid} \hat{s} +\hat{S}_f^{\rm high} \hat{s}.\]
\end{proposition}
The proof is given in Appendix \ref{proofs of g}.
Proposition \ref{Prop_SfLTFT} shows that $S_f$ is a linear non-negative filter. Proposition \ref{Prop_SfLTFT} also gives an explicit formula for $S_f^{-1}$ as the operator that multiplies by $\frac{1}{\hat{S}_f(z)}$ in the frequency domain. The integrals (\ref{eq:Sf_hat}) can be numerically estimated pre-processing and saved as part of the LTFT transform, and used each time the transform is applied on a signal. The theory guarantees that $\hat{S}_f(z)$ is stably invertible. In practice, for reasonable windows like the Hann window, $\hat{S}_f(z)$ is typically close to a constant function with small disturbances about the transition frequencies $a,b$.

\section{Discrete stochastic time-frequency signal processing}
\label{Discretization of phase space in time frequency analysis}

In this subsection we present two discretizations under which the CWT, the STFT and the LTFT are linear volume discretizable (Definition \ref{D:linear area discretizable}). \textcolor{black}{By Theorem \ref{Main_Theorem12}, this means that  the number of Monte Carlo samples required for a given error tolerance is only linear in the resolution of the discrete signal, and do not depend directly on the dimension of phase space.} 

In the first discretization we consider the following setting. We analyze time signals $s:\RR\rightarrow\CC$ by decomposing them to compact time interval sections. Without loss of generality, we suppose that each signal segment is supported in $[-1/2,1/2]$. Indeed, the restriction of the signal to any compact intervals can be transformed by an affine linear change of variables to the support $[-1/2,1/2]$. We consider two  regimes for segmenting the signal. One option is to restrict the signal $s:\RR\rightarrow\CC$ to finite intervals $\{I_k\subset\RR\}_{k\in\ZZ}$ to obtain $s_k = s|_{I_k}$, and suppose that 
\begin{equation}
s=\sum_{k\in\ZZ} s_k.
\label{eq:part_sig}
\end{equation}
 Another option is to consider a partition of unity $\{\xi_k\in L^{\infty}(\RR)\}_{k\in\ZZ}$ where each $\xi_k$ is supported on a compact interval, positive in its support, and $\sum_{k\in\ZZ}\xi_k(x)=1$ for every $x\in\RR$. We then consider the signal segments $s_k=\xi_k s$, and observe that (\ref{eq:part_sig}) is satisfied. We call the multiplication of $s$ by $\xi_k$ \emph{enveloping}. In either of these two regimes, we carry out the time-frequency analysis for each signal segment separately, assuming it is supported in $[-1/2,1/2]$.
For the STFT and LTFT we  discretize all of $L^2(\RR)$ directly.

\subsection{Discrete stochastic CWT}
\label{Stochastic CWT signal processing}

Consider a CWT based on an admissible mother wavelet $f\in L_2(\RR)$
 with compact time support $[-S,S]$ for some $S>0$.
For the CWT we consider the partition of unity regime, where our signal segment $q$ is defined as $q(x)=\xi(x)s(x)$ and both $\xi$ and $s$ are supported in $(-1/2,1/2)$. We assume that $\xi$ is non-negative, continuously differentiable, and obtains zero only outside $(-1/2,1/2)$.

We prove linear discretization for the following class of signals. 
\begin{definition}
Let $C>0$. The class $\mathcal{R}_{C}$ is the set of all signals $q\in L^2[-1/2,1/2]$ such that
\begin{equation}
\norm{\xi^{-1}q}_{\infty}< C\norm{q}_{\infty}
\label{eq:RC1}
\end{equation}
and
\begin{equation}
\quad \norm{q}_{\infty} < C\norm{q}_2.
\label{eq:RC2}
\end{equation}
\end{definition}

\begin{remark}
We interpret $\mathcal{R}_{C}$ as follows. 
\begin{enumerate}
	\item 
	\emph{Equation (\ref{eq:RC1}) assures that enveloping $s$ by $\xi$ does not eliminate most of the content of $s$.}
	To see this, by $q=\xi s$, equation (\ref{eq:RC1}) can be written as
\begin{equation}
\norm{s}_{\infty}< C\norm{\xi s}_{\infty}.
\label{eq:RC10}
\end{equation}
Enveloping $s$ with $\xi $ can in principle eliminate the content of $s$ near $-1/2$ and $1/2$, since $\xi $ is zero there. The signal $s$ could approach $\infty$ at $-1$ and $1$, but $q$ would be zero there, so multiplying $s$ by $\xi $ discards most of the content of $s$. However, (\ref{eq:RC10}) assures that enveloping $s$ with $\xi $ does not do that.
\item
\emph{Equation (\ref{eq:RC2}) assures that the energy of $q$ is well spread on the interval $[-1/2,1/2]$.} Indeed, no small subset of $[-1/2,1/2]$ can contain most of the energy of $q$, otherwise $\norm{q}_{\infty}$ would be significantly larger than $\norm{q}_2$.
\end{enumerate}
\end{remark}

We consider the following discretization of $L^2(-1/2,1/2)$. For each $M\in\NN$,
\begin{equation}
V_M= {\rm span}\{e^{2 \pi i m x}\xi(x)\}_{m=-M}^M.
\label{eq:VMCWT}
\end{equation}
Namely, $V_M$ is the space of enveloped trigonometric polynomials of order $M$. It is easy to see that $V_M$ is indeed a discretization of $L^2(-1/2,1/2)$. We moreover have the following result.
\begin{proposition}
\label{VmRc_disc}
 The sequence of spaces $\{V_M\cap \mathcal{R}_{C}\}_{M\in\NN}$ is a discretization of $\mathcal{R}_{C}$. 
\end{proposition}
The proof is in Appendix \ref{Proofs of e}.

Let $W>0$.
For each $M\in\NN$ we consider the following phase space domain $G_M\subset G$, where $G$ is the wavelet time-frequency plane, represented by frequency $\w=\tau^{-1}$ instead of scale $\tau$ (see Subsection \ref{The 1D continuous wavelet transform})
\begin{equation}
G_M = \big\{(x,\w)\ \big|\ W^{-1}M^{-1} < \abs{\w} <WM ,\ \abs{x}< 1/2+ S/\w\big\}.
\label{eq:GM_CWT}
\end{equation}
The area of $G_M$ in the time-frequency plane is, for large enough $M$,
\begin{equation}
\mu(G_M) = 2\int_{W^{-1}M^{-1}}^{WM}(1+ 2S/\w)d\w \leq 2WM + 8S\ln\Big(WM\Big) \leq 3 WM.
\label{eq:VolGM_CWT}
\end{equation}
Denote 
\begin{equation}
\psi_M(g) = \left\{\begin{array}{ccc}
	1 & , & g\in G_M \\
	0 & ,  & g\notin G_M.
\end{array}\right.
\label{eq:psi_CWT}
\end{equation}

\begin{proposition}
\label{lin_vol_CWT}
Consider a smooth enough $\xi$ in the sense
\begin{equation}
\hat{\xi}(z) \leq \left\{ 
\begin{array}{ccc}
	D         & , & \abs{z}\leq 1  \\
	D z^{-k} & , & \abs{z}>1 
\end{array}
\right.
\label{eq:xi_hat_bound}
\end{equation}
for some $k>2$ and $D>0$, and the corresponding discrete spaces $\{V_M\}_{M\in\NN}$ of (\ref{eq:VMCWT}).
The continuous wavelet transform with a compactly supported mother wavelet $f\in L_2(\RR)$ is linear volume discretizable with respect to the class $\mathcal{R}_{C}$ and the discretization $\{V_M\cap \mathcal{R}_C\}_{M\in\NN}$, with the envelopes $\psi_M$ defined by (\ref{eq:psi_CWT}) for large enough $W$ that depends only on $\e$ of Definition \ref{D:linear area discretizable}. 
\end{proposition}
The proof is in Appendix \ref{Proofs of e}.

\subsection{\textcolor{black}{Linear volume discretization of STFT and LTFT}}
\label{Linear volume discretization of the LTFT}

We start with the STFT. Let $f\in L^2(\RR)$ be a window function supported in $[-S,S]$ for some $S>0$. Let $C',Y,\k>1/2$, and suppose that $f$ satisfies
\begin{equation}
\text{for every }\abs{z}>Y, \quad \hf(z) \leq C'\abs{z}^{-\k}.
\label{eq:fSTFT_decay}
\end{equation}
Consider the STFT based on $f$. We construct the following discretization of $L^2(\RR)$. 
\begin{definition}
\label{de:LTFT_disc}
For every $R,M\in\NN$, Define $V_{M,R}$ as the space of signal $q\in L^2(\RR)$ supported in the time interval $[-R/2,R/2]$, where in $[-R/2,R/2]$, $q$ is a trigonometric polynomial of order $M$
\begin{equation}
    \label{eq:D_trig}
    \forall x\in [-R,R], \quad q(x) = \sum_{m=-M}^M c_n R^{-1/2}\exp\big(\frac{2\pi i}{R} n x\big),
\end{equation}
for some $\mathbf{c}=\{c_m\}_{m=-M}^M$.
\end{definition}
 Note that in Definition \ref{de:LTFT_disc}, $\norm{\mathbf{c}}_2=\norm{q}_2$.
The following proposition is direct.
\begin{proposition}
For any sequence $M_J,R_J$ such that $R_J=o(M_J)$ and \newline
$ M_J,R_J \xrightarrow[J \to \infty]{} \infty$, the sequence of spaces $\{V_{M_J,R_J}\}_{J=1}^{\infty}$ is a discretization of $L^2(\RR)$.
\end{proposition}

In the space $V_{M,R}$, we interpret $M/R$ as the \emph{fidelity}, or the \emph{sampling rate}.
When $R=\Theta(M)$, e.g., $R=  M/B'$ for some $B'>0$, the fidelity $B'$ does not go to infinity in the asymptotic analysis. In this case, we call $V_{M,R}$ the \emph{band-limited regime}, as every trigonometric polynomial in $V_{M,R}$ has frequencies in the fixed band $B'$.

Define the phase space domain
\begin{equation}
G^W_{M,R}=G_{M,R}:=[-R/2-S/2,R/2+S/2]\times [-WM/R,WM/R].
    \label{eq:VolGM_CWT20}
\end{equation}
and denote
\begin{equation}
\psi^W_{M,R}(g)=\psi_{M,R}(g) := \left\{\begin{array}{ccc}
	1 & , & g\in G^W_{M,R} \\
	0 & ,  & g\notin G^W_{M,R}.
\end{array}\right.
\label{eq:psi_STFT}
\end{equation}
The following proposition extends the STFT LVD property of \cite[Theorem 29]{Ours1} from compact intervals to all of $L^2(\RR)$.

\begin{proposition}
\label{prop:LVD_STFT}
Under the above setting, the STFT with window $f$ supported in $[-S,S]$ and satisfying (\ref{eq:fSTFT_decay}) is LVD with respect to the discretization $V_{M,R}$, with $R=O(M)$, with the envelops $\psi^W_{M,R}$ defined in (\ref{eq:psi_STFT}), for large enough $W$ that depends only on $\e$ of Definition \ref{D:linear area discretizable}.
\end{proposition}
The proof is in Appendix \ref{Proofs of h}.


Next, we formulate the LVD result of the LTFT with more flexibility than Definition \ref{D:linear area discretizable}, allowing the transition frequencies $a_{\t},b_{\t}$ to depend on the discretization. 
Consider the sequence of LTFTs $\{V_f^{M,R}\}_{M,R\in\NN}$, with $f$ supported at $(-1/2,1/2)$ and satisfying (\ref{eq:fSTFT_decay}), and with transition frequencies of (\ref{eq:LTFT_atom}), depending on $M,R$, at $0<a_0<a_{\tau}^{M,R}<b_{\tau}^{M,R}<M/R$, where $a_0$ is some global constant.
It is useful to add this flexibility to the asymptotic analysis, since, when the fidelity of the discrete signal $M/R$ tends to infinity, we would like to allow the high transition frequency $b_{\tau}$ of (\ref{eq:LTFT_atom}) to be proportional to $M/R$.  
The idea is that one typically adjusts the fidelity of the discrete space to the content of the signal -- high frequencies should capture information, and not the `tail' of the signal in the frequency domain.  We would like the bulk of the frequency content to be analyzed by the CWT atoms of the LTFT, and the high frequencies should be analyzed with STFT atoms only to alleviate transient artifact. Hence, we typically choose $b_{\tau}=\beta M/R$ for some $0<\beta<1$.

We now define the envelopes of the LTFT  as follows.
Let $W>0$.
For each $R<M\in\NN$ we define
\begin{equation}
\begin{split}
G^W_{M,R} =G_{M,R} :=& \big\{(x,\w,\tau)\ \big|\ \tau_1\leq \tau\leq \tau_2,\ -WM/ R < \w < WM/ R , \\
& \quad  \quad  \quad   \quad   \quad   -R/2- \tau_2/a^{M,R}_{\tau} <x< R/2+ \tau_2/a^{M,R}_{\tau} \big\}.
\end{split}
\label{eq:GM_LTFT}
\end{equation} 
Recall that the measure $\mu_{\tau}$ along the $\tau$ axis satisfies  $\mu_{\tau}([\tau_1,\tau_2])=1$. Thus,
the area of $G_{M,R}$ is the time-frequency space is
\begin{equation}
\mu(G_{M,R}) = WMO(1).
\label{eq:VolGM_LTFT}
\end{equation}
Denote 
\begin{equation}
\psi^W_{M, R}(g) =\psi_{M, R}(g): = \left\{\begin{array}{ccc}
	1 & , & g\in G^W_{M,R} \\
	0 & ,  & g\notin G^W_{M,R}.
\end{array}\right.
\label{eq:psi_LTFT}
\end{equation}
The following proposition shows that the LTFT is LVD in the extended asymptotic analysis.

\begin{proposition}
\label{Other_class_LTFT}
Under the above construction of the LTFT and $V_{M,R}$, for every $\epsilon>0$ there is a large enough $W>0$ such that for every $s_{M,R}\in V_{M,R}$ with $R<M$,
\begin{equation}
\frac{\norm{V^{M,R}_f[s_{M,R}] - \psi^W_{M,R} V^{M,R}_f[s_{M,R}]}_2}{\norm{V^{M,R}_f[s_{M,R}]}_2} < \e,
\label{eq:LVD_LTFT}
\end{equation}
where $\norm{\psi^W_{M,R}}_1 = WO\big({\rm dim}(V^{M,R})\big)$.
\end{proposition}

The proof is in Appendix \ref{Proofs of h}. The following corollary states the special case where the LTFT is fixed in the asymptotic analysis.

\begin{corollary}
\label{Other_class_LTFT2}
Under the above setting of the LTFT and $V_{M,R}$, with $f$  satisfying (\ref{eq:fSTFT_decay}),
the LTFT $V_f$ is linear volume discretizable with respect to  $\{V_{M,R}\}_{M\in\NN}$, with $R=O(M)$, and with the envelopes $\psi^W_{M, R}$ defined by (\ref{eq:psi_LTFT}) for large enough $W$ that depends only on $\e$ of Definition \ref{D:linear area discretizable}.
\end{corollary}

\subsection{Stochastic LTFT phase vocoder}
\label{Stochastic signal processing with the localizing time-frequency transform}

In this subsection we analyze the example of the stochastic phase vocoder  \textcolor{black}{(Example \ref{Phase vocoder example})} based on the LTFT. \textcolor{black}{For phase vocoder methods see \cite[Section 7.4.3]{vocoder_book} and \cite{New_vocoder0,New_vocoder1,ltfatnote050,Ottosen2017APV}.}
An advantage in using LTFT atoms instead of STFT atoms is for alleviating artifacts such as transient smearing, echo,
and loss of presence \cite{Ours1}. These artifacts are manifestations of  phasiness \cite{phasiness0},  \textcolor{black}{the audible artifact resulting from summing two time-frequency atoms with intersecting time and frequency supports, but with out of sync phases. In \cite{phasiness0} the phenomenon was described as follows: 
\begin{quote}
    ``Phasiness or reverberation or loss
of presence relates to the fact that the modified signal often sounds as if it had been recorded in a small room. In
particular, time-expanded speech sounds like the speaker is
much further from the microphone than it was in the original
sound.''
\end{quote}}
\textcolor{black}{The LTFT was suggested in \cite{Ours1} as a way to alleviate such artifacts by using wavelet atoms, which typically have shorter supports and less time overlaps than STFT atoms, and by adding the ``number of oscillations axis,'' which allows representing both transient events (impulse-like features) and harmonic parts.}
 Sound examples and code of the LTFT  phase vocoder are available at \url{https://github.com/RonLevie/LTFT-Phase-Vocoder}.


\subsubsection{Formulation of stochastic LTFT phase vocoder}

Given a real valued time signal $s\in L^2(\RR)$, the values of $\hs(\w)$ at the negative frequencies $\w<0$ are uniquely determined by $\hs(\w)$ for $\w>0$ \textcolor{black}{due to the Hermitian symmetry \cite{Asignal}.} Thus, in practice, we consider only LTFT atoms with $\w>0$. After the signal processing pipeline, we post-process the output signal by taking its real part and multiplying by 2.
\textcolor{black}{Moreover, in practice we sample from the phase space domain $G_{M,R}^W$ of (\ref{eq:GM_LTFT}), with $W\geq 1$ close to 1, e.g., $G^1_{M,R}$ with the phase space frequency support  $[-M/R,M/R]$. We choose such a discretization   
 even though the $\e$ error of (\ref{eq:lin_area2}) in this case is not guaranteed to be uniformly small in $s$. This choice of $G_{M,R}^W$ is reasonable when assuming that the fidelity $M/R$ is high enough so that the signal content of $s\in V_{M,R}$  due to atoms with frequencies $\w>M/R$ is negligible.}

Similarly to Example \ref{Phase vocoder example} and in the notations of (\ref{eq:LTFT_atom}), let $\{(x_k,\w_k,\tau_k)\}_{k=1}^K$ be random independent uniform samples from 
\begin{equation}
    \tilde{G}_{M,R}=[-R/2,R/2]\times [0,M/R]\times [\tau_1,\tau_2].
    \label{eq:DLTFT_G}
\end{equation}
 Let $\Delta\in\NN$ be the dilation integer. We define the stochastic integer time dilation phase vocoder as follows.
The synthesis formulation is
\begin{equation}
s\mapsto    C\sum_{k=1}^K r\big(V_f [S_f^{-1}s](x_k,\w_k,\tau_k)\big)f_{\Delta x_k,\w_k,\tau_k}.
\label{eq:PSSPS22MC012}
\end{equation}
The analysis formulation is
\begin{equation}
s\mapsto C S_f^{-1} \sum_{k=1}^K r\big(V_f[s](x_k,\w_k,\tau_k)\big)f_{\Delta x_k,\w_k,\tau_k}.
\label{eq:PSSPA2MC012}
\end{equation}
where $r$ is defined by $r(e^{i\theta}a)=e^{i\Delta\theta}a$, for $a,\theta\in\RR_+$ as in  Example \ref{Phase vocoder example}, and the normalization is $C=\frac{M(\tau_2-\tau_1)}{2K}$.

\subsubsection{\textcolor{black}{Computational complexity of stochastic LTFT phase vocoder}}
\label{Computational complexity of LTFT phase vocoder}

In this subsection we compute the computational complexity of a digital implementation of the LTFT phase vocoder.
We digitize the LTFT as follows.
The  signals that we consider are supported at the time interval $[-R/2,R/2]$ with sampling rate $M/R$, so we digitize them to time series with time samples $\{\frac{Rn}{M}\}_{n=-M/2}^{M/2}$. We denote the space of such digital signals by $D_{M,R}$. 

Since the method is stochastic, we compute the mean number of floating-point  operations in the end-to-end pipeline.
Suppose that $R\in\NN$ is even. 
Let $0<\a<\b<1$, and suppose that the transition frequencies satisfy 
\[a^{M,R}\leq\a M/R , \quad b^{M,R}=\b M/R,\]
and consider the phase space domain $\tilde{G}_{M,R}$ of (\ref{eq:DLTFT_G}). 
 Define the average number of oscillations as $\tau_0 = \int_{\tau_1}^{\tau_2}\tau d\tau$.
Let $K=ZM$ be the number of Monte Carlo samples in (\ref{eq:PSSPS22MC012}) or (\ref{eq:PSSPA2MC012}). 
The number of scalar operations entailed by each atom $f_{x,\w,\tau}$, either due to the inner product $\ip{s}{f_{x,\w,\tau}}$ or due to the sum in (\ref{eq:PSSPS22MC012}) or (\ref{eq:PSSPA2MC012}),   is estimated as the time support of the atom times the sampling rate. 
In the following, we list the resulting average number of floating-point  operations performed in the different bands, given random $x$, $\tau$ and $\w$ in the corresponding band.

\begin{enumerate}
	\item \textbf{Low STFT.}
Time support of the atom:  $\frac{\tau}{a^{M,R}}$.
Number of time samples in the atom: $\frac{M\tau}{R a^{M,R}}$.
Probability of sampling low STFT atoms: $\frac{Ra^{M,R}}{M}$.
Average number of low windows sampled: $\frac{KRa^{M,R}}{M}$.
Overall average number of operations: $\frac{KRa^{M,R}}{M} \frac{M\tau}{R a^{M,R}}= K\tau_0$.
\item
\textbf{High STFT. }
Time support of the atom:  $\frac{\tau}{b^{M,R}}=\frac{\tau R}{\b M}$.
Number of time samples in the atom:  $\frac{M\tau}{R b^{M,R}}=\frac{\tau}{ \b }$.
Probability of sampling high STFT atoms: $1-\b$.
Average number of high windows sampled: $K(1-\b)$.
Overall average number of operations in the high band:
$\frac{\tau_0K(1-\b)}{\b}$.
\item
\textbf{Middle CWT. }
Time support of the atom (for frequency $\w$):  $\tau/\w$.
Number of time samples in the CWT atoms: $\frac{M\tau}{R\w}$.
Average number of time samples in CWT atoms:
\[\begin{split}
&\frac{1}{\b M/R - a^{M,R}}\int_{\tau_1}^{\tau_2}\int_{a^{M,R}}^{\b M/R}\frac{M\tau}{R\w} d\w d\tau \\
& = \frac{M\tau_0}{\b M - Ra^{M,R}}\ln(\w)\big|_{a^{M,R}}^{\b M/R} \\
 & =  \frac{\tau_0}{\b  - a^{M,R}R/M}\Big(\ln(M) + \ln(\b) - \ln(R) - \ln(a^{M,R}) \Big).
\end{split}\]
%
Probability of sampling CWT atoms: $(\b - a^{M,R}R/M)$.
Average number of CWT atoms sampled: $K(\b - a^{M,R}R/M)$.
Overall average number of operations:
$\tau_0K\ln\Big(\ln(M) + \ln(\b) - \ln(R) - \ln(a^{M,R}) \Big)$. 
In the special case where $a^{M,R} = \a M/R$, the overall average number of operations is
$\tau_0K\ln\Big(\frac{\b}{\a}\Big)$.
\end{enumerate}
\begin{proposition}
\label{prop:FLOP}
For $K=ZM$, the expected number of floating-point  operations performed by the LTFT phase vocoder is
\[\begin{split}
   & \mathbb{E}(\#{\rm operations}) \\
   & = \tau_0ZM\Big(1 +\frac{(1-\b)}{\b} + \ln(M) + \ln(\b) - \ln(R) - \ln(a^{M,R})\Big) + O(M\log M).
\end{split}\]
When $a^{M,R} = \a M/R$, the expected number of floating-point  operations is
\[\mathbb{E}(\#{\rm operations}) = 2\tau_0ZM\Big(1 +\frac{(1-\b)}{\b} + \ln\big(\frac{\b}{\a}\big)\Big) + O(M\log M).\]
\end{proposition}
In Proposition \ref{prop:FLOP}, the term $O(M\log M)$ is due to the number of operations entailed by applying $S_f^{-1}$ via FFT.

\bibliographystyle{spmpsci}	
\bibliography{Ref_uncertainty3,RandNLA}

\appendix

\section{Pseudo inverse of the analysis operator}
\label{Pseudo inverse of teh analsis operator}

For an injective linear operator with close range $B:\mathcal{V}\rightarrow \mathcal{W}$ between the Hilbert spaces $\mathcal{V}$ and $\mathcal{W}$, we define the pseudo inverse \cite[Section 2.1.2]{PsInv}
\[B^+:\mathcal{W}\rightarrow \mathcal{V}, \quad B^+ = \big(B\big|_{\mathcal{V}\rightarrow B\mathcal{V}}\big)^{-1}R_{B\mathcal{V}},\]
where $R_{B\mathcal{V}}:\mathcal{W}\rightarrow B\mathcal{V}$ is the surjective operator given by the orthogonal projection from $\mathcal{W}$ to $B\mathcal{V}$ and restriction of the image space to the range $B\mathcal{V}$,  
 and $B\big|_{\mathcal{V}\rightarrow B\mathcal{V}}$ is the restriction of the image space of $B$ to its range $B\mathcal{V}$, in which it is invertible.
 Note that $R_{B\mathcal{V}}^*$ is the operator that takes a vector from $B\mathcal{V}$ and canonically embeds it in $\mathcal{W}$, and $P_{B\mathcal{V}}=R_{B\mathcal{V}}^*R_{B\mathcal{V}}:\mathcal{W}\rightarrow\mathcal{W}$ is the orthogonal projection upon $B\mathcal{V}$. 
\textcolor{black}{Note that $V_f^+$ exists. Indeed, by the frame inequality (\ref{eq:FB}), $V_f$ is bounded both from above and below, so it must be injective with closed range \cite[Chapter 2]{PsInv0}. 
%
In the following, we collect basic properties
 from frame analysis (see, e.g., \cite{Cframe1}, \cite[Section 5.6]{IntFrame} and \cite[Section 2.1.2]{PsInv})} 


\begin{lemma}
\label{Lemm_PI}
Let $f:G\rightarrow\cH$ be a continuous frame with frame bounds $A,B$. Then the following properties hold. Let $s\in \cH$.
\begin{enumerate}
\item 
	\label{Lemm_PI:2}
	The operator $V_f^*$ is the pseudo inverse of $V_f^{+*}$, and $V_f^{+*}[\cH]=V_f[\cH]$.
	\item
	\label{Lemm_PI:20}
	$V_f^{+*}V_f^*=V_fV_f^+=P_{V_f[\cH]}$.
		\item 
	\label{Lemm_PI:3}
	$S_f^{-1}=(V_f^*V_f)^{-1}= V_f^+V_f^{+*}$.
  \item 
	\label{Lemm_PI:0}
	$V_f^{+*}[s]=V_{S_f^{-1}f}[s]=V_{f}[S_f^{-1}s]$.
	\item 
	\label{Lemm_PI:1}
	$V_f^{+}=S_f^{-1}V_f^*$.
	\item
	\label{Lemm_PI:10}
	$\norm{V_f^+}_2\geq A^{-1/2}$.
\end{enumerate}
\end{lemma}

\begin{proof}
We prove \ref{Lemm_PI:1}, and note that the rest are basic properties of dual frames and pseudo inverse. See, e.g., \cite{Cframe1} and \cite[Section 5.6]{IntFrame} for dual frames, and \cite[Section 2.1.2]{PsInv} for pseudo inverse. 
\[S_f^{-1}V_f^*  = V_f^{+}V_f^{+*}V_f^* = V_f^{+}P_{V_f[\cH]}  =V_f^{+}.
\]
\end{proof}


\section{Proofs}

\subsection{Proofs of Section  \ref{The localizing time-frequency transform}}
\label{proofs of g}

The mapping $(x,\w,\tau)\mapsto f_{x,\w,\tau}$ is continuous \textcolor{black}{as a mapping $[\tau_1,\tau_2]\times\RR^2\rightarrow L^2(\RR)$, since $\mathcal{T}(x),\mathcal{M}(\w),\mathcal{D}(\gamma)$ are continuous, in $x,\w,$ and $\gamma$ respectively, in the strong topology \cite[Section 9.2]{Time_freq}}. Hence, $V_f[s]:G\rightarrow\CC$ is a continuous function for every $s\in L^2(\RR)$, and thus measurable. To show that $f$ is continuous frame, it is left to show the existence of frame bounds $0<A\leq B<\infty$ satisfying (\ref{eq:FB}). Equivalently we show that $V_f$ is injective and
\begin{equation}
\norm{V_f} \leq B^{1/2}, \quad  \norm{V_f^{+}} \leq A^{-1/2}
\label{eq:FB_LTFT}
\end{equation}
where $V_f^{+}:L^2(G)\rightarrow L^2(\RR)$ is the pseudo inverse of $V_f$ as defined in Subsection \ref{Non-Parseval phase space signal processing}.

We start by deriving a formula for the LTFT atoms $\hf_{x,\w,\tau}$ in the frequency domain.

\begin{lemma}
\label{LTFT_atom_TF}
LTFT atoms take the following form in the frequency domain.
\begin{equation}
\begin{split}
\hf_{x,\w,\tau}(z) 
  & = \left\{
\begin{array}{ccc}
	\mathcal{M}(-x) \mathcal{T}(\w) \mathcal{D}(a_{\tau}/\tau) \hf(z)  & {\rm if} & \abs{\w}<a_{\tau} \\
	\mathcal{M}(-x) \mathcal{T}(\w)  \mathcal{D}(\w/\tau) \hf(z)  & {\rm if} & a_{\tau}\leq\abs{\w}\leq b_{\tau} \\
	\mathcal{M}(-x) \mathcal{T}(\w) \mathcal{D}(b_{\tau}/\tau)\hf(z)   & {\rm if} &  b_{\tau}<\abs{\w}
\end{array}
\right. \\
 & =
\left\{
\begin{array}{ccc}
	\sqrt{\frac{\tau}{a_{\tau}}}\hf\big(\frac{\tau}{a_{\tau}}(z-\w)\big)e^{-2\pi i x z} & {\rm if} & \abs{\w}<a_{\tau} \\
	\sqrt{\frac{\tau}{\w}}\hf\big(\frac{\tau}{\w}(z-\w)\big)e^{-2\pi i x z} & {\rm if} & a_{\tau}\leq\abs{\w}\leq b_{\tau} \\
	 \sqrt{\frac{\tau}{b_{\tau}}}\hf\big(\frac{\tau}{b_{\tau}}(z-\w)\big)e^{-2\pi i x z} & {\rm if} & b_{\tau}<\abs{\w}.
\end{array}
\right.
\end{split}
\label{eq:LTFT_atom_freq}
\end{equation}
For $f_{\tau}(t) = \tau^{-1/2}f(\tau^{-1}t)e^{2\pi i t}$ we have
\begin{equation}
\quad \hf_{\tau}(z) =\tau^{1/2}\hf\big(\tau (z-1)\big),
\label{eq:ps_mother}
\end{equation}
and another formula in case $a_{\tau}\leq\abs{\w}\leq b_{\tau}$ is
\[\hf_{x,\w,\tau}(z) =  \w^{-1/2}\hf_{\tau}(\w^{-1}z)e^{-2\pi i x z}.\]
\end{lemma}

\begin{proof}
Equation (\ref{eq:LTFT_atom_freq}) is a direct result of Lemma \ref{Transform_lemma} and (\ref{eq:LTFT_atom}).
We can write $f_{\tau} = \mathcal{M}(1) \mathcal{D}(\tau)h$.
Another formula in case $ a_{\tau}\leq\abs{\w}\leq b_{\tau} $ is
\[f_{x,\w,\tau} = \mathcal{T}(x)\mathcal{D}(\w^{-1})\hf_{\tau},\]
so
\[\hf_{x,\w,\tau} =  \mathcal{M}(-x)\mathcal{D}(\w)\hf_{\tau}.\]
We can also write
\[\hat{f}_{x,\w,\tau} =\mathcal{M}(-x)\mathcal{D}(\w) \mathcal{T}(1) \mathcal{D}(\tau^{-1})\hf.\]
\end{proof}

For convenience, we repeat here Definition \ref{sub-frame filters}.
The sub-frame filter kernels $\hat{S}_f^{\rm low}$, $\hat{S}_f^{\rm mid}$ and $\hat{S}_f^{\rm high}$ are the functions $\RR\rightarrow\CC$ defined by
\begin{equation}
\hat{S}_f^{\rm band}(z) =  \int_{{\tau}_1}^{{\tau}_2}  \int_{\mathcal{J}^{\rm band}_{\tau}}     \abs{\overline{\hf^{\rm band}(\tau,\w; z-\w)}}^2  d\w   d{\tau}.
\label{eq:Sf_hat22}
\end{equation}
The frame filter kernel  $\hat{S}_f:\RR\rightarrow\CC$ is defined by
\begin{equation}
\hat{S}_f = \hat{S}_f^{\rm low}+\hat{S}_f^{\rm mid}+\hat{S}_f^{\rm high}.
\label{eq:Sf_hat3}
\end{equation}
Let ${\rm band}\in \{{\rm low},{\rm mid},{\rm high}\}$. 
 For convenience, we recall equation (\ref{eq:LTFT_atom_freq_hh}) 
\begin{equation}
\hf^{\rm band}(\tau,\w; z-\w) = \left\{
\begin{array}{ccc}
	\sqrt{\frac{\tau}{a_{\tau}}}\hf\big(\frac{\tau}{a_{\tau}}(z-\w)\big) & {\rm if} & \abs{\w}<a_{\tau} \\
	\sqrt{\frac{\tau}{\w}}\hf\big(\frac{\tau}{\w}(z-\w)\big) & {\rm if} & a_{\tau}\leq\abs{\w}\leq b_{\tau} \\
	 \sqrt{\frac{\tau}{b_{\tau}}}\hf\big(\frac{\tau}{b_{\tau}}(z-\w)\big) & {\rm if} & b_{\tau}\leq\abs{\w}.
\end{array}
\right.
\label{eq:LTFT_atom_freq_hh2}
\end{equation}
Thus, by (\ref{eq:LTFT_atom_freq}),
\[\hf_{x,\w,\tau}(z) = \hf^{\rm band}(\tau,\w; z-\w)e^{-2\pi i x z},\]
for the band corresponding to $\w$.

\begin{proof}[Proof of Theorem \ref{LTFT_frame}]
Denote by $V_f^{\rm band}[s]$ the restriction of $V_f[s]$ to $(x,\w,\tau)$ satisfying $\w\in \mathcal{J}_{\tau}^{\rm band}$.
 We offer the following informal computation.
\[
\begin{split}
 & \norm{V_f^{\rm band}[s]}_2^2 \\ &   = \int_{{\tau}_1}^{{\tau}_2}\int_{\mathcal{J}^{\rm band}_{\tau}} \int_{\RR} \abs{V_f[s](x,\w,{\tau}) }^2 dx d\w  d{\tau}\\
 & = \int_{{\tau}_1}^{{\tau}_2}\int_{\mathcal{J}^{\rm band}_{\tau}} \int_{\RR}  \int_{y\in\RR} \hs(y)\overline{\hat{f}_{x,\w,{\tau}}(y)} dy \int_{z\in\RR} \overline{\hs(z)} \hf_{x,\w,{\tau}}(z) dz dx d\w  d{\tau} \\
  & = \int_{{\tau}_1}^{{\tau}_2}\int_{\mathcal{J}^{\rm band}_{\tau}} \int_{\RR} \int_{y\in\RR} \hs(y)\overline{\hf^{\rm band}(\tau,\w; y-\w)e^{-2\pi i x y}} dy \\
  & \quad\quad\quad\quad\quad\quad\quad\quad\quad\quad\quad\quad\quad\quad\quad\quad\quad\quad\quad\quad  \int_{z\in\RR}  \overline{\hs(z)}\hf^{\rm band}(\tau,\w; z-\w)e^{-2\pi i x z}  dzdx d\w  d{\tau} \\
  & = \int_{{\tau}_1}^{{\tau}_2}\int_{\mathcal{J}^{\rm band}_{\tau}}   \int_{z\in\RR}\int_{y\in\RR}\Big(\int_{\RR} e^{-2\pi i x (z-y)} dx\Big) \hs(y)\overline{\hf^{\rm band}(\tau,\w; y-\w)}  \\
  & \quad\quad\quad\quad\quad\quad\quad\quad\quad\quad\quad\quad\quad\quad\quad\quad\quad\quad\quad\quad \overline{\hs(z)}\hf^{\rm band}(\tau,\w; z-\w) dy dy d\w  d{\tau} \\
  & = \int_{{\tau}_1}^{{\tau}_2}\int_{\mathcal{J}^{\rm band}_{\tau}}  \int_{z\in\RR} \int_{y\in\RR}  \delta(z-y) \hs(y)\overline{\hf^{\rm band}(\tau,\w; y-\w)}  \overline{\hs(z)} \hf^{\rm band}(\tau,\w; z-\w) dy dz d\w  d{\tau} \\
  & = \int_{{\tau}_1}^{{\tau}_2}\int_{\mathcal{J}^{\rm band}_{\tau}}  \int_{z\in\RR}   \hs(z)\overline{\hf^{\rm band}(\tau,\w; z-\w)}  \overline{\hs(z)} \hf^{\rm band}(\tau,\w; z-\w) dz  d\w  d{\tau} 
\end{split}
\]
\begin{equation}
= \int_{z\in\RR}\abs{\hs(z)}^2 \int_{{\tau}_1}^{{\tau}_2}  \int_{\mathcal{J}^{\rm band}_{\tau}}     \abs{\overline{\hf^{\rm band}(\tau,\w; z-\w)}}^2  d\w   d{\tau}dz .
\label{eq:nomr_fram_op}
\end{equation}
Here, $\delta$ is the delta functional, and the informal computation with the delta functional can be formulated appropriately similarly to the usual Calderón’s reproducing formula in continuous wavelet analysis (see, e.g., \cite[Proposition 2.4.1 and 2.4.1]{Ten_lectures}, \cite[Theorem 1]{Calderon1}, and \cite[Theorem 2.5]{Calderon2}).

Now note that by the fact the three $\mathcal{J}^{\rm band}_{\tau}$ domains are disjoint, so
\[\norm{V_f[s]}_2^2 = \norm{V_f^{\rm low}[s]}_2^2+\norm{V_f^{\rm mid}[s]}_2^2+\norm{V_f^{\rm high}[s]}_2^2.\]
Thus, by (\ref{eq:nomr_fram_op}) and (\ref{eq:Sf_hat3}),
\[\norm{V_f[s]}_2^2  = \int_{z\in\RR}\abs{\hs(z)}^2\hat{S}_f(z) dz. \]

Our goal now is to show that $\hat{S}_f(z)$ is bounded from below by some $A>0$ and from above by some $B>0$ for every $z\in\RR$. The constants $A,B$ are the frame bounds.
In the following we construct implicit upper and lower bounds for $\hat{S}_f(z)$, without any effort to make these bound realistic estimates of $\norm{V_f}^2$ and $\norm{V_f^{+}}^2$. The goal is to prove that $f$ is a continuous frame, rather than to obtain good frame bounds. In Subsection \ref{The frame operator of the LTFT} we explain separately that numerically estimating $\hat{S}_f$ and $\hat{S}_f^{-1}$ give good estimates for the frame bounds.

Next, we show that there is some $A>0$ such that for every $z\in\RR$ $\hat{S}_f(z)\geq A$.
For simplicity, we consider the case where $a_{\tau}=a$ and $b_{\tau}=b$ are constants. The general case is shown similarly with the appropriate modifications.
Let $z\geq 0$, and note that the case $z\leq 0$ is shown symmetrically. 
By the fact that $f$ is a non-negative function,
\begin{equation}
\hf(0) = \norm{f}_1>0.
\label{eq:hh0}
\end{equation}
Since $f$ is compactly supported, $\hf$ is smooth, so there is some $\nu>0$ such that for every $z\in(-\nu,\nu)$
\begin{equation}
\hf(z) \geq \frac{1}{2}\norm{f}_1=:C_0.
\label{eq:hh01}
\end{equation}
\textcolor{black}{We now distinguish between three cases: $z\in [0,a]$, $z\in \left.\left(a,b\right.\right]$, and $z\in (b,\infty)$.}

In case
$z\in [0,a]$, 
\begin{equation}
   \hf^{\rm band}(\tau,\w; z-\w) = \sqrt{\frac{\tau}{a}}\hf\big(\frac{\tau}{a}(z-\w)\big). 
   \label{eq:ert77777g}
\end{equation}
\textcolor{black}{By lugging (\ref{eq:hh01}) in (\ref{eq:ert77777g}),} for any $\w$ satisfying
\[\w \in (z-\nu\frac{a}{\tau_2}, z+\nu\frac{a}{\tau_2})\]
we have
\begin{equation}
\abs{\hf^{\rm band}(\tau,\w; z-\w)} \geq  \sqrt{\frac{\tau_1}{a}}C_0.
\label{eq:SfLow1}
\end{equation}
Let $\mathcal{I}_z$ denote the interval $(z-\nu\frac{a}{\tau_2}, z+\nu\frac{a}{\tau_2})\cap(-a,a)$, and note that the length of $\mathcal{I}_z$ is bounded from below by
\begin{equation}
\mu(\mathcal{I}_z)\geq \nu\frac{a}{\tau_2},
\label{eq:length_J_1}
\end{equation}
\textcolor{black}{where $\mu$ is the standard Lebesgue measure or $\RR$.}
Thus, by the fact that the integrand of (\ref{eq:Sf_hat}) is non-negative,  the fact that $\mu_{\tau}([\tau_1,\tau_s])=1$, and by (\ref{eq:SfLow1}) and (\ref{eq:length_J_1}),
\[\hat{S}_f^{\rm band}(z) = \int_{{\tau}_1}^{{\tau}_2}  \int_{\mathcal{J}^{\rm band}_{\tau}}     \abs{\overline{\hf^{\rm band}(\tau,\w; z-\w)}}^2  d\w   d{\tau} \]
\begin{equation}
\geq  \int_{{\tau}_1}^{{\tau}_2}  \int_{\mathcal{I}_z}     \abs{\overline{\hf^{\rm band}(\tau,\w; z-\w)}}^2  d\w   d{\tau}
\geq \nu\frac{a}{\tau_2}\frac{\tau_1}{a}C_0^2 = C_1.
\label{eq:S_f+low11}
\end{equation}

If $z\in \left.\left(a,b\right.\right]$, 
\[\hf^{\rm band}(\tau,\w; z-\w) = \sqrt{\frac{\tau}{\w}}\hf\big(\frac{\tau}{\w}(z-\w)\big). \]
By (\ref{eq:hh01}), for any $\w$ satisfying
\[\w \in (z-\nu\frac{a}{\tau_2}, z+\nu\frac{a}{\tau_2})\]
we have
\begin{equation}
\abs{\hf^{\rm band}(\tau,\w; z-\w)} \geq  \sqrt{\frac{\tau_1}{b}}C_0.
\label{eq:SfLow2}
\end{equation}
Let $\mathcal{I}_z$ denote the interval $(z-\nu\frac{a}{\tau_2}, z+\nu\frac{a}{\tau_2})\cap(a,b)$, and note that the length of $\mathcal{I}_z$ is bounded from below by
\begin{equation}
\mu(\mathcal{I}_z)\geq \nu\frac{a}{\tau_2}.
\label{eq:length_J_2}
\end{equation}
Thus, by the fact that the integrand of (\ref{eq:Sf_hat}) is non-negative,  by $\mu_{\tau}([\tau_1,\tau_s])=1$, (\ref{eq:SfLow2}) and (\ref{eq:length_J_2}),
\[\hat{S}_f^{\rm band}(z) = \int_{{\tau}_1}^{{\tau}_2}  \int_{\mathcal{J}_{\tau}}     \abs{\overline{\hf^{\rm band}(\tau,\w; z-\w)}}^2  d\w   d{\tau} \]
\begin{equation}
\geq  \int_{{\tau}_1}^{{\tau}_2}  \int_{\mathcal{I}_z}     \abs{\overline{\hf^{\rm band}(\tau,\w; z-\w)}}^2  d\w   d{\tau}
\geq \nu\frac{a}{\tau_2}\frac{\tau_1}{b}C_0^2 = C_2.
\label{eq:S_f+low22}
\end{equation}

Last, if $z\in (b,\infty)$, 
\[\hf^{\rm band}(\tau,\w; z-\w) = \sqrt{\frac{\tau}{b}}\hf\big(\frac{\tau}{b}(z-\w)\big). \]
By (\ref{eq:hh01}), for any $\w$ satisfying
\[\w \in (z-\nu\frac{b}{\tau_2}, z+\nu\frac{b}{\tau_2})\]
we have
\begin{equation}
\abs{\hf^{\rm band}(\tau,\w; z-\w)} \geq  \sqrt{\frac{\tau_1}{b}}C_0.
\label{eq:SfLow3}
\end{equation}
Let $\mathcal{I}_z$ denote the interval $(z-\nu\frac{b}{\tau_2}, z+\nu\frac{b}{\tau_2})\cap(b,\infty)$, and note that the length of $\mathcal{I}_z$ is bounded from below by
\begin{equation}
\mu(\mathcal{I}_z)\geq \nu\frac{b}{\tau_2}.
\label{eq:length_J_3}
\end{equation}
Thus, by the fact that the integrand of (\ref{eq:Sf_hat}) is non-negative,  by $\mu_{\tau}([\tau_1,\tau_s])=1$, (\ref{eq:SfLow3}) and (\ref{eq:length_J_3}),
\[\hat{S}_f^{\rm band}(z) = \int_{{\tau}_1}^{{\tau}_2}  \int_{\mathcal{J}_{\tau}}     \abs{\overline{\hf^{\rm band}(\tau,\w; z-\w)}}^2  d\w   d{\tau} \]
\begin{equation}
\geq  \int_{{\tau}_1}^{{\tau}_2}  \int_{\mathcal{I}_z}     \abs{\overline{\hf^{\rm band}(\tau,\w; z-\w)}}^2  d\w   d{\tau}
\geq \nu\frac{b}{\tau_2}\frac{\tau_1}{b}C_0^2 = C_3.
\label{eq:S_f+low33}
\end{equation}
By taking $A=\min\{C_1,C_2,C_3\}$, for every $z\geq 0$
\[\hat{S}_f^{\rm band}(z)  \geq A,\]
and thus
\[\norm{V_f[s]}_2^2\geq A\norm{s}_2^2.\]

Next, we bound $\norm{V_f}^2$ from above. Note that
\[\norm{V_f} =\norm{V_{f^{\rm low}}}+\norm{V_{f^{\rm mid}}}+\norm{V_{f^{\rm high}}},\]
where $V_{f^{\rm band}}$ now denotes $V_f$ restricted to $(x,\w,\tau)$ with $\w\in\mathcal{J}^{\rm band}_{\tau}$, for any ${\rm bad}\in\{{\rm low, mid, high}\}$.
The systems $f^{\rm low}$ and $f^{\rm high}$ are both STFT systems restricted in phase space to a sub-domain of frequencies, and integrated along $\tau\in(\tau_1,\tau_2)$. By extending $f^{\rm low}$ and $f^{\rm high}$ to the whole frequency axis $\RR$, we increase $\norm{V_{f^{\rm low}}}$ and $\norm{V_{f^{\rm high}}}$ to the frame bound of the STFT which is 1. This shows that
\[\norm{V_{f^{\rm low}}}, \ \norm{V_{f^{\rm high}}} \leq 1.\]
It is left to bound $\norm{V_{f^{\rm mid}}}^2$ from above. Note that $f^{\rm mid}$ cannot be extended to a CWT frame, since this system is not based on an admissible wavelet.

Recall the pseudo mother wavelet defined in  (\ref{eq:false_mother})
\[f_{\tau}(t) = \tau^{-1/2}f(\tau^{-1}t)e^{2\pi i t} .\]
In the following we use the bound
\begin{equation}
\norm{\hf_{\tau}}_{\infty} \leq \norm{f_{\tau}}_1 = \tau^{1/2}\int \abs{f(\tau^{-1}t)} \tau^{-1}dt = \tau^{1/2}\norm{h}_1 \leq \tau_2^{1/2}\norm{h}_1 \leq \tau_2^{1/2}\norm{h}_2 = C_0.
\label{eq:ss6kwp9}
\end{equation}

By (\ref{eq:nomr_fram_op})
\[\norm{V_{f^{\rm mid}}}^2= \int_z  \abs{\hs(z)}^2\hat{S}_{f}^{\rm mid}(z) dz,\]
and by Lemma \ref{LTFT_atom_TF},
\begin{equation}
\hat{S}_{f}^{\rm mid}(z)=\int_{\tau_1}^{\tau_2}\int_{\w\in[-b,-a]\cap[a,b]}\w^{-1}\abs{\hf_{\tau}(\w^{-1}z)}^2d\w d\tau.
\label{eq:s55yr}
\end{equation}
Let us change variable in (\ref{eq:s55yr}) and consider the interval $\w\in[a,b]$. By
$\w^{-1}z = y$, we have
$\w = zy^{-1} $
, 
$d\w = -z y^{-2}dy$, 
and 
\[\w=a \Leftrightarrow y=a^{-1}z , \quad \w=b \Leftrightarrow y=b^{-1}z.\]
Thus
\[\hat{S}_{f}^{\rm mid}(z)=\int_{\tau_1}^{\tau_2}\int_{b^{-1}z}^{a^{-1}z}yz^{-1}\abs{\hf_{\tau}(y)}^2z y^{-2}dy d\tau =\int_{\tau_1}^{\tau_2}\int_{b^{-1}z}^{a^{-1}z}\abs{\hf_{\tau}(y)}^2 y^{-1}dy d\tau.\]
Therefore, by (\ref{eq:ss6kwp9}), 
\[
\begin{split}
 \hat{S}_{f}^{\rm mid}(z) & =\int_{\tau_1}^{\tau_2}\int_{b^{-1}z}^{a^{-1}z}\abs{\hf_{\tau}(y)}^2 y^{-1}dy d\tau  \leq \int_{\tau_1}^{\tau_2}C_0 \int_{b^{-1}z}^{a^{-1}z} y^{-1}dy d\tau \\
 &  = \int_{\tau_1}^{\tau_2}C_0 \ln\Big(\frac{b}{a}\Big) d\tau =(\tau_2-\tau_1)C_0 \ln\Big(\frac{b}{a}\Big)=C'.
\end{split}
\]

To conclude,
\[\norm{V_f} \leq \norm{V_{f^{\rm low}}}+\norm{V_{f^{\rm mid}}}+\norm{V_{f^{\rm high}}} \leq 2+\sqrt{C'} =: \sqrt{B}.\]

\end{proof}

\begin{proof}[Proof of Proposition \ref{Prop_SfLTFT}]
For any ${\rm band}\in \{{\rm low},{\rm mid},{\rm high}\}$, \textcolor{black}{define the operator $S_f^{\rm band}$ by}
\[ S_f^{\rm band} s =\int_{{\tau}_1}^{{\tau}_2}  \int_{\mathcal{J}^{\rm band}_{\tau}}\int_{\RR} V_f[s](x,\w,\tau) dx d\w d\tau,\]
and observe that
\[ S_f=S_f^{\rm low}+S_f^{\rm mid}+S_f^{\rm high}.\]
We show
that for any ${\rm band}\in \{{\rm low},{\rm mid},{\rm high}\}$,  and any $s\in L^2(\RR)$,
\begin{equation}
\cF S_f^{\rm band} s (z)= \hat{S}_f^{\rm band}(z)\hs(z).
\label{eq:Prop_SfLTFT}
\end{equation}

 We offer the following informal computation,  analogues to the proof of Theorem \ref{LTFT_frame}.
\[
\begin{split}
 \cF S_f^{\rm band} s (z) &   = \int_{{\tau}_1}^{{\tau}_2}\int_{\mathcal{J}^{\rm band}_{\tau}} \int_{\RR} V_f[s](x,\w,{\tau}) \hat{f}_{x,\w,{\tau}}(z) dx d\w  d{\tau}\\
 & = \int_{{\tau}_1}^{{\tau}_2}\int_{\mathcal{J}^{\rm band}_{\tau}} \int_{\RR}  \int_{\RR} \hs(y)\overline{\hat{f}_{x,\w,{\tau}}(y)} dy  \hat{f}_{x,\w,{\tau}}(z) dx d\w  d{\tau} \\
 & = \int_{{\tau}_1}^{{\tau}_2}\int_{\mathcal{J}^{\rm band}_{\tau}} \int_{\RR}  \int_{\RR} \hs(y)\overline{\hf^{\rm band}(\tau,\w; y-\w)e^{-2\pi i x y}} dy\\
 & \quad\quad\quad\quad\quad\quad\quad\quad\quad\quad\quad\quad\quad\quad\quad\quad\quad\quad\quad\quad \hf^{\rm band}(\tau,\w; z-\w)e^{-2\pi i x z} dx d\w  d{\tau} \\
 & = \int_{{\tau}_1}^{{\tau}_2}\int_{\mathcal{J}^{\rm band}_{\tau}}   \int_{\RR}  \Big(\int_{\RR} e^{-2\pi i x (z-y)} dx\Big) \\
 & \quad\quad\quad\quad\quad\quad\quad\quad\quad\quad\quad\quad\quad\quad \hs(y)\overline{\hf^{\rm band}(\tau,\w; y-\w)}   \hf^{\rm band}(\tau,\w; z-\w) dy d\w  d{\tau} \\
 & = \int_{{\tau}_1}^{{\tau}_2}\int_{\mathcal{J}^{\rm band}_{\tau}}   \int_{\RR}  \delta(z-y) \hs(y)\overline{\hf^{\rm band}(\tau,\w; y-\w)}   \hf^{\rm band}(\tau,\w; z-\w) dy d\w  d{\tau} \\
 &  = \int_{{\tau}_1}^{{\tau}_2}\int_{\mathcal{J}^{\rm band}_{\tau}}     \hs(z)\overline{\hf^{\rm band}(\tau,\w; z-\w)}   \hf^{\rm band}(\tau,\w; z-\w)  d\w  d{\tau}\\
 &  = \hs(z) \int_{{\tau}_1}^{{\tau}_2}  \int_{\mathcal{J}^{\rm band}_{\tau}}     \abs{\overline{\hf^{\rm band}(\tau,\w; z-\w)}}^2  d\w   d{\tau} .
\end{split}
\]
Here, $\delta$ is the delta functional, and the formal computation with the delta functional can be formulated appropriately as explained in the proof of Theorem \ref{LTFT_frame}.

\end{proof}

\subsection{Proofs of Subsection \ref{Stochastic CWT signal processing}}
\label{Proofs of e}

\begin{proof}[Proof of Proposition \ref{VmRc_disc}]
Let $\d$. 
Given $q\in \mathcal{R}_C$, we can approximate $q$ by a smooth function $p\in L^2(-1/2,1/2)\cap L^{\infty}(-1/2,1/2)$ that vanishes in a neighborhood of $-1/2$ and $1/2$, 
 up to the small errors
\begin{equation}
\norm{q-p}_{2}<\d, \quad \abs{\ \norm{q}_{\infty}-\norm{p}_{\infty} } < \d, \quad \abs{\ \norm{\xi^{-1}q}_{\infty}-\norm{\xi^{-1}p}_{\infty} } < \d.
\label{eq:VmRc_disc1}
\end{equation}
Since $p$ and $\xi^{-1} p$ have continuously differentiable periodic extensions, their Fourier series converge to $p$ and $\xi^{-1} p$ respectively in both $L^2(-1/2,1/2)$ and $L^{\infty}(-1/2,1/2)$  \textcolor{black}{\cite[Section 4.4]{Gibbs}}. 
We denote by $v_M$ the truncation of the Fourier series of $\xi^{-1}p$ up to the frequency $M$,  multiplied by $\xi$.   \textcolor{black}{Namely,
\[v_M(x)= \sum_{m=-M}^M\ip{ \xi^{-1}p }{ e^{2\pi i  m(\cdot)} }\xi(x)e^{2\pi i  mx}.\]
}
There is thus a large enough $M$,  such that $v_M\in V_M$ satisfies
\begin{equation}
\norm{v_M-p}_{2}<\d, \quad \norm{v_M-p}_{\infty} < \d, \quad \norm{\xi^{-1}v_M-\xi^{-1}p}_{\infty} < \d.
\label{eq:VmRc_disc2}
\end{equation}
Given any $\d'>0$,
by choosing $\d$ small enough, and $M$ large enough, equations (\ref{eq:VmRc_disc1}) and (\ref{eq:VmRc_disc2}) guarantee that $v_M\in V_M\cap\mathcal{R}_C$ and $\norm{v_M-q}_2<\d'$.

\end{proof}

We prove Proposition \ref{lin_vol_CWT} in a sequence of claims.
Consider the phase space domain $G_M$ of (\ref{eq:GM_CWT}).
Recall that $[-S,S]$ is the support of the mother wavelet $f$. Thus, $[-S/\w,S\w]$ is the support of the dilated wavelet $\w^{1/2}f(\w z)$.
Since the signal $q$ is supported in time in $[-1/2,1/2]$, $V_f[q](x,\w)$ is zero for any $x\notin (-1/2-S/\w,  1/2+ S/\w)$. As a result, restricting $V_f[q]$ to the phase space domain
\begin{equation}
   G_M' = \big\{(x,\w)\ \big|\ W^{-1}M^{-1} < \w <WM \big\}
   \label{eq:new_G'}
\end{equation}
is equivalent to restricting $V_f[q]$ to $G_M$. We thus consider without loss of generality the domain $G_M'$ instead of $G_M$ in this section. 

We define the following inner product that corresponds to enveloping by $\xi$. 
\begin{definition}[\textcolor{black}{Weighted Lebesgue space}]
\label{xi_inner_prod}
For any two measurable $q,p:[-1/2,1/2]\rightarrow\CC$ 
\[\ip{q}{p}_{\xi } =\ip{\xi^{-1}q}{\xi^{-1}p} = \int_{-1/2}^{1/2} \frac{1}{\xi^2(x)}q(x)\overline{p(x)}dx\]
where $\ip{\xi^{-1}q}{\xi^{-1}p}$ is the $L^2[-1/2,1/2]$ inner product.
Denote
\[\norm{q}_{\xi }= \sqrt{\ip{q}{p}_{\xi }}.\]
Denote by $L_{2;\xi }(-1/2,1/2)$ the Hilbert space of signals with $\norm{q}_{\xi}<\infty$.
\end{definition}

The following lemma characterizes the behavior of admissible wavelets about the zero frequency.
\begin{lemma}
\label{moment_f}
Let $f\in L_2(\RR)$ be an admissible wavelet supported in $[-S,S]$. Then
\[\abs{\hf(z)} \leq 2\pi S\norm{f}_1 \abs{z} \leq 2^{3/2}\pi S^{3/2}\norm{f}_2\abs{z}.\]
\end{lemma}

\begin{proof}
First note that by the fact that $f$ is supported in $(-S,S)$, it is also in $L^1(\RR)$ by the Cauchy–Schwarz inequality, with
\[\norm{f}_1 \leq \sqrt{2S}\norm{f}_2.\]
By the wavelet admissibility condition $\hf(0)=0$. Moreover, since $f$ is compactly supported, $\hf$ is smooth. Thus, for $z>0$,
\[\hf(z) = \int_0^z \hf'(y)dy.\]
As a result
\[\abs{\hf(z)} \leq \int_0^z \abs{\hf'(y)}dy \leq \norm{\hf'}_{\infty}z \leq \norm{2\pi yf(y)}_{1}z \leq 2S\pi\norm{f}_{1}z \leq 2^{3/2}\pi S^{3/2}\norm{f}_2\abs{z},\]
\textcolor{black}{where $\norm{\hf'}_{\infty} \leq \norm{2\pi yf(y)}_{1}$ since $(\cF\hf')(y)=2\pi i yf(y)$, and 
\[\norm{\hf'}_{\infty} = \sup_{\w\in\RR}\abs{\int_{\RR} 2\pi i yf(y) e^{-2\pi i \w y} dy} \leq \sup_{\w\in\RR}\int_{\RR} 2\pi \abs{yf(y)}  dy = \norm{2\pi yf(y)}_{1}.\]} 
For $z<0$ the proof is similar.

\end{proof}

Note that the basis $\{e^{2\pi i m x}\xi(x)\}_{m=-M}^M$  of $V_M$ is an orthonormal \textcolor{black}{system} in $L_{2;\xi}(-1/2,1/2)$, \textcolor{black}{since $\{e^{2\pi i m x}\}_{m=-M}^M$  are orthonormal  in $L_{2}(-1/2,1/2)$}.
By Parseval's identity,
for $q\in V_M$ satisfying
\[q(x) = \sum_{m=-M}^M c_n e^{2\pi i m x}\xi(x),\]
we have
\[\norm{q}_{\xi} = \sqrt{\sum_{m=-M}^M \abs{c_m}^2} = \norm{\{c_m\}_{m=-M}^M}_2.\]

We prove Proposition \ref{lin_vol_CWT} by embedding the signal class $\mathcal{R}_C$ in a richer space, and proving linear volume discretization for the richer space.
Consider the signal space $\mathcal{S}_E$ of signals $q\in L^2(-1/2,1/2)$ satisfying
\begin{equation}
\norm{q}_{\xi} \leq E\norm{q}_2
\label{eq:SE}
\end{equation}
for some fixed $E>0$.

\begin{lemma}
\label{RC_sub_SE}
$\mathcal{R}_C\subset\mathcal{S}_E$ for any $E\geq C^2$.
\end{lemma}

\begin{proof}
Let $q=\xi s \in \mathcal{R}_C$. By Cauchy Schwartz inequality \textcolor{black}{and by (\ref{eq:RC1}) and (\ref{eq:RC2})}
\[\norm{q}_{\xi } = \norm{s}_2 \leq \norm{s}_{\infty}= \norm{\xi^{-1}q}_{\infty} \leq   C \norm{q}_{\infty} \leq  C^2 \norm{q}_2\]
so
\[\norm{q}_{\xi } \leq  C^2 \norm{q}_2 \leq E \norm{q}_2.\]
\end{proof}

\begin{proposition}
\label{ClaimLinVolCWT}
Under the above construction, with $\xi$ satisfying (\ref{eq:xi_hat_bound}) with $k>2$, we have $\mu(G_M)\leq 5WM$, and for every $q_M\in V_M\cap\mathcal{S}_E$ we have
\begin{equation}
\frac{\norm{I-\psi_M V_f[q_M]}_2}{\norm{V_f[q_M]}_2} = O(W^{-1})+o_M(1),
\label{eq:ClaimLinVolCWT}
\end{equation}
\textcolor{black}{where the O notation $O(W^{-1})$ is with respect to $W\rightarrow\infty$, and $o_M(1)$ is a function that goes to zero as $M\rightarrow\infty$.}
\end{proposition}

\textcolor{black}{Next, we prove Proposition \ref{lin_vol_CWT}, which we copy here for the convenience of the reader. \newline
\textbf{Proposition \ref{lin_vol_CWT}} 
\textit{Consider a smooth enough $\xi$ in the sense
\begin{equation}
\hat{\xi}(z) \leq \left\{ 
\begin{array}{ccc}
	D         & , & \abs{z}\leq 1  \\
	D z^{-k} & , & \abs{z}>1 
\end{array}
\right.
\label{eq:xi_hat_bound222}
\end{equation}
for some $k>2$ and $D>0$, and the corresponding discrete spaces $\{V_M\}_{M\in\NN}$ of (\ref{eq:VMCWT}).
The continuous wavelet transform with a compactly supported mother wavelet $f\in L_2(\RR)$ is linear volume discretizable with respect to the class $\mathcal{R}_{C}$ and the discretization $\{V_M\cap \mathcal{R}_C\}_{M\in\NN}$, with the envelopes $\psi_M$ defined by (\ref{eq:psi_CWT}) for large enough $W$ that depends only on $\e$ of Definition \ref{D:linear area discretizable}. }
}

By Lemma \ref{RC_sub_SE}, Proposition \ref{lin_vol_CWT} is now a corollary of Proposition \ref{ClaimLinVolCWT}, where for every $\e$ we choose $W$ and $M_0$ large enough, so that for every $M>M_0$
\begin{equation}
\frac{\norm{I-\psi_M V_f[q_M]}_2}{\norm{V_f[q_M]}_2} < \e.
\label{eq:ClaimLinVolCWT2}
\end{equation}

\begin{proof}[Proof of Proposition \ref{ClaimLinVolCWT}]
Since $\{e^{2\pi i m x}\xi(x)\}_{m=-M}^M$  is an orthonormal basis of $V_M\subset L_{2;\xi}(-1/2,1/2)$,
in the frequency domain signals in $V_M$ are spanned by $\{\hb_m(z)=\hat{\xi}(z-m)\}_{m=-M}^M$ (see Lemma \ref{Transform_lemma}). 
We bound $\hb_n(z)$ by
\[\abs{\hb_m(z)} = \abs{\hat{\xi}(z-m)}\leq 
\left\{ 
\begin{array}{ccc}
	D         & , & \abs{z-m}\leq 1  \\
	D \abs{z-m}^{-k} & , & \abs{z-m}>1 .
\end{array}
\right.
\]
For $\hq = \sum_{m=-M}^M c_m \hb_m$, with $\norm{q}_{\xi}=\norm{\{c_m\}_{m=1}^N}_2$, we have $\abs{c_m}\leq \norm{q}_{\xi}$ for all $m$. Thus
\begin{equation}
\abs{\hq(z)} \leq \sum_{m=-M}^M \abs{c_m} \abs{\hb_m(z)}
\leq \sum_{m=-M}^M \norm{q}_{\xi}
\left\{ 
\begin{array}{ccc}
	D         & , & \abs{z-m}\leq 1  \\
	D \abs{z-m}^{-k} & , & \abs{z-m}>1 
\end{array}
\right.
\label{eq:LinVolCWT101}
\end{equation}
For $\abs{z}\leq M+2$, (\ref{eq:LinVolCWT101}) leads to the bound
\begin{equation}
\abs{\hq(z)}
\leq 2D\norm{q}_{\xi}+ \norm{q}_{\xi}\sum_{-M\leq m\leq M, m\neq \left\lceil z \right\rceil ,\left\lfloor z\right\rfloor}
	D \abs{z-m}^{-k},
\label{eq:LinVolCWT1}
\end{equation}
and (\ref{eq:LinVolCWT1}) 
is maximized by taking $z=0$.
 We extend the sum to $m$ between $-\infty$ and $\infty$, and without loss of generality increase the value of the sum by choosing $z=0$. Thus, since $k>2$,
\[
\begin{split}
 \abs{\hq(\w)} & \leq 2D\norm{q}_{\xi}+ 2\norm{q}_{\xi} D\sum_{m=1}^{\infty} m^{-k} \leq  2D\norm{q}_{\xi}+ 2\norm{q}_{\xi} D \Big( 1 + \int_1^{\infty}m^{-k} dm \Big)\\
 &  \leq  4\norm{q}_{\xi} D + 2\norm{q}_{\xi} D(k-1)^{-1} \leq 6 D \norm{q}_{\xi}.  
\end{split}
\]
Now, for $\abs{z}>M+2$, without loss of generality consider $z>0$. By (\ref{eq:LinVolCWT101}) we have
\[
\begin{split}
 \abs{\hq(z)} & \leq \norm{q}_{\xi} \sum_{m=-M}^M D (z-m)^{-k} \leq  \norm{q}_{\xi} \int_{-\infty}^{M+1} D (z-m)^{-k} dm\\
 &  \leq    \norm{q}_{\xi}  D (z-M-1)^{-k+1}(k-1)^{-1}.
\end{split}
\]
Overall, 
\begin{equation}
\frac{\abs{\hat{q(z)}}}{\norm{q}_{\xi}} \leq \hat{E}(z)
:= \left\{ 
\begin{array}{ccc}
	6 D          & , & \abs{z}\leq M+2  \\
	6 D  \big(\abs{z}-M-1\big)^{-k+1}& , & \abs{z}> M+2 
\end{array}
\right. .
\label{eq:hqInftyqT}
\end{equation}

We consider the domain $G_M' = \{(x,\w)\ |\ W^{-1}M^{-1} < \w <AM \}$.
Recall that restricting to $G_M'$ is equivalent to restricting to $G_M$. 
Let $\psi_M$ denote by abuse of notation the projection in phase space that restricts functions to the domain $G_M'$.
Next, we bound the error $\norm{(I-\psi_M)V_f[q]}_2$ for signals in $V_M$. 
Note that for every $\w\in\RR$
\begin{equation}
\int_{\RR}\abs{V_f[q](\w,x)}^2 dx = \int_{\RR} \abs{\hq(z)}^2\w^{-1}\abs{\hf(\w^{-1}z)}^2  dz.
\label{eq:Vfs_slice_CWT}
\end{equation}
Indeed, by Lemma \ref{Transform_lemma}
\[
\begin{split}
 V_f[q](x,\w)& =  \int_{\RR} \hq(z)\w^{-1/2}\overline{\hf(\w^{-1}z) e^{-2\pi i x z}}  dz \\
 &   =\cF^{-1}\Big(\hq \w^{-1/2}\overline{\hf\big(\w^{-1}(\cdot)\big)}\Big) (x),
\end{split}
\]
so by Plancherel's identity
\begin{equation}
\int_{\RR}\abs{V_f(q)(\w,x)}^2 dx = \int_{\RR} \abs{\hq(z) \w^{-1/2}\overline{\hf\big(\w^{-1}z\big)}}^2 dz.
\label{eq:PlanchW}
\end{equation}
Let us now bound the right-hand-side of (\ref{eq:PlanchW})
 for $\w$ such that $(x,\w)\notin G_M'$.
We then integrate the result for $\w\in (WM,\infty)$, and for $\w\in (0,W^{-1}M^{-1})$, to bound the error in phase space truncation by $\psi_M$. Negative $\w$ are treated similarly.

We start with $\w>WM$. Here, we decompose the integral along $z$ into two integrals with boundaries in $0, M+2 + \big(0.5(\w-M-2)\big)^{1/2}$. For each segment of the integral along $z$, by additivity of the integral, we integrate $\w$ along $(WM,\infty)$ and show that the resulting value is small.
The first segment gives
\[\int_0^{M+2 + \big(0.5(\w-M-2)\big)^{1/2}} \abs{\hq(z)}^2\abs{\w^{-1/2}\hf(\w^{-1} z)}^2dz\]
\[\leq \int_0^{M+2 + \big(0.5(\w-M-2)\big)^{1/2}} \abs{\hq(z)}^2 dz \max_{0\leq z <M+2 + \big(0.5(\w-M-2)\big)^{1/2}}\abs{\w^{-1/2}\hf(\w^{-1} z)}^2.\]
By Lemma \ref{moment_f}, the dilated mother wavelet satisfies
\begin{equation}
\w^{-1/2}\hf(\w^{-1} z) \leq 
	2\pi S \norm{f}_1 \abs{z} \w^{-1.5} ,
\label{eq:Bdilf}
\end{equation}
so
\begin{equation}
\begin{split}
 & \int_{0}^{M+2 + \big(0.5(\w-M-2)\big)^{1/2}} \abs{q(z)}^2\abs{\w^{-1/2}\hf(\w^{-1} z)}^2dz \\
 &   \leq \norm{q}_2^2 4\pi^2 S^2 \norm{f}_1^2 \Big(M+2 + \big(0.5(\w-M-2)\big)^{1/2}\Big)^{2} \w^{-3} = 4\pi^2 S^2 \norm{q}_2^2 \norm{f}_1^2  G(\w)
\end{split}
\label{eq:LinVolCWT12}
\end{equation}
for $G(\w) = \Big(M+2 + \big(0.5(\w-M-2)\big)^{1/2}\Big)^{2} \w^{-3}$.
We study $G(\w)$ for different values of $\w$. 
When $\w> M^2$, the value $\Big(M+2 + \big(0.5(\w-M-2)\big)^{1/2}\Big)^{2}$ is bounded by $4\w$ for large enough $M$, so 
\begin{equation}
G(\w) \leq 4\w^{-2}.
\label{eq:Ber1}
\end{equation}
When $\w< M^2$, $\Big(M+2 + \big(0.5(\w-M-2)\big)^{1/2}\Big)^{2}$ is bounded by $4M^{2}$ for large enough $M$, so
\begin{equation}
G(\w) \leq 4M^{2}\w^{-3}.
\label{eq:Ber2}
\end{equation}
For $\w>M^2$, integrating via the bound (\ref{eq:Ber1}) gives
\begin{equation}
\int_{M^2}^{\infty}\int_{0}^{M+2 + \big(0.5(\w-M-2)\big)^{1/2}} \abs{q(z)}^2\abs{\w^{-1/2}\hf(\w^{-1} z)}^2dz d\w =o_M(1)\norm{q}_2^2,
\label{eq:LinCWT30}
\end{equation}
where $o_M(1)$ is a function that converges to zero as $M\rightarrow\infty$.
For $WM<\w< M^2$, integrating via the bound (\ref{eq:Ber2}) gives
\begin{equation}
\int_{WM}^{M^2}\int_{0}^{M+2 + \big(0.5(\w-M-2)\big)^{1/2}} \abs{q(z)}^2\abs{\w^{-1/2}\hf(\w^{-1} z)}^2dz d\w =O(W^{-2})\norm{q}_2^2
\label{eq:LinCWT20}
\end{equation}
Overall,  (\ref{eq:LinCWT20}) and (\ref{eq:LinCWT30}) give
\begin{equation}
\int_{WN}^{\infty}\int_{0}^{M+2 + \big(0.5(\w-M-2)\big)^{1/2}} \abs{\hq(z)}^2\abs{\w^{-1/2}\hf(\w^{-1} z)}^2dz d\w =\big(O(W^{-2}) + o_M(1)\big)\norm{q}_2^2.
\label{eq:LinCWT2}
\end{equation}

Next, we study $z\in \Big(  M+2 + \big(0.5(\w-M-2)\big)^{1/2},\infty  \Big)$. Here, we take the maximum of the signal squared, and take the 2 norm of the window. By (\ref{eq:hqInftyqT}) we obtain for large enough $M$
\begin{equation}
\begin{split}
 & \int_{M+2 + \big(0.5(\w-M-2)\big)^{1/2}}^{\infty} \abs{q(z)}^2\abs{\w^{-1/2}\hf(\w^{-1} z)}^2dz \\
 &  \leq\norm{f}_2^2 36D^2 \Big(M+2 + \big(0.5(\w-M-2)\big)^{1/2}- M-1\Big)^{-2k+2} \norm{q}_{\xi}^2 \\
  & < \norm{f}_2^2 36D^2 2(\w-M-2)^{-k+1} \norm{q}_{\xi}^2
\end{split}
\label{eq:LinVolCWT13}
\end{equation}
Integrating the bound (\ref{eq:LinVolCWT13}) along $w\in(WM,\infty)$ gives
\begin{equation}
\int_{WM}^{\infty}\int_{M+2 + \big(0.5(\w-M-2)\big)^{1/2}}^{\infty} \abs{q(z)}^2\abs{\w^{-1/2}\hf(\w^{-1} z)}^2dz d\w = o_M(1)\norm{q}_{\xi}^2=o_M(1)\norm{q}_2^2,
\label{eq:LinCWT3}
\end{equation}
since $k>2$ and $\norm{q}_{\xi}\leq E\norm{q}_2$.

Last, we integrate $\w\in (0,W^{-1}M^{-1})$. By (\ref{eq:hqInftyqT})
\[\norm{\hat{q}}_{\infty} \leq 6D\norm{q}_{\xi}.\]
Thus
\begin{equation}
\int_{-\infty}^{\infty}\abs{\hat{q}(z)}^2\abs{\w^{-1/2}\hf(\w^{-1}z)}^2dz
\leq \norm{\hat{q}}_{\infty}^2\norm{\hf}^2_2 \leq  36D^2 \norm{f}_2^2 \norm{q}_{\xi}^2 = O(1) \norm{q}_2^2.
\label{eq:lastLinVolCWT}
\end{equation}
Thus, the integration of the bound (\ref{eq:lastLinVolCWT}) for $\w\in (0,W^{-1}M^{-1})$ gives
 \begin{equation}
\int_0^{W^{-1}M^{-1}}\int_{-\infty}^{\infty}\abs{\hat{q}(z)}^2\abs{\w^{-1/2}\hf(\w^{-1}z)}^2dz = O(M^{-1})\norm{q}_{2}^2.
 \label{eq:LinCWT4}
 \end{equation}

Summarizing the estimates (\ref{eq:LinCWT2}), (\ref{eq:LinCWT3}) and (\ref{eq:LinCWT4}), together with the analogue bounds for $\w<0$, we obtain
\[\norm{(I-\psi_M)V_f[q]}_2= \big(O(W^{-1})  + o_M(1)\big)\norm{q}_2.\] 
Last, since by the frame assumption $\norm{q}_2^2 \leq A^{-1}\norm{V_f(q)}_2^2$, we have
\[\frac{\norm{(I-\psi_M)V_f[q]}_2}{\norm{V_f[q]}_2}= \mathcal{J}(W,M)= O(W^{-1})  + o_M(1).\] 
This means that given $\e>0$, we may choose $W$ large enough to guarantee $\mathcal{J}(W,M) < \e$ up from some large enough $M_0$, and also guarantee for every $M\in\NN$
\[\norm{\psi_M}_1\leq C^{\e} M\] 
with $C^{\e}= 3W$ by (\ref{eq:VolGM_CWT}).

\end{proof}

\subsection{\textcolor{black}{Proofs of Subsection \ref{Linear volume discretization of the LTFT}}}
\label{Proofs of h}

Recall that every $q\in V_{M,R}$ is
a trigonometric polynomials supported on $L^2(-R/2,R/2)$ 
\[q(x) = \sum_{m=-M}^M c_n R^{-1/2}\exp\big(\frac{2\pi i}{R} n x\big)\]
with $\norm{\mathbf{c}}_2=\norm{q}_2$.
Since the Fourier transform of the indicator function of $[-1/2,1/2]$ is the sinc function
${\rm sinc}(z) := \frac{\sin(\pi x)}{\pi x}$,
the normalized indicator function of $[-R/2,R/2]$ is given in the frequency domain by
\[R^{1/2}{\rm sinc}(Rz) \leq \frac{R^{-1/2}}{\pi\abs{z}}. \]
Hence,
\begin{equation}
   \hq(z)=R^{1/2}\sum_{n=-M}^M c_n {\rm sinc}\big(R(z-R^{-1}n)\big) \label{eq:Fq}
\end{equation}
%
%
%
%
recall that (\ref{eq:fSTFT_decay}) reads:  for every $z>Y$ or $z<-Y$
\begin{equation}
\hf(z) \leq C'\abs{z}^{-\k}.
\label{eq:fSTFT_decay2}
\end{equation}
Recall that STFT envelope $G_{M,R}$ of (\ref{eq:VolGM_CWT20}) is defined by
\begin{equation}
G_{M,R}=[-R/2-S/2,R/2+S/2]\times [-WM/R,WM/R].
    \label{eq:VolGM_CWT2}
\end{equation}
In the proof of Proposition \ref{prop:LVD_STFT} we use the  following simple fact, that can be shown by a direct calculation.
\begin{lemma}
\label{lem:STFT_LVD1}
Let $V_f$ be the STFT based on the window $f\in L^2(\RR)$, and let $q\in L^2(\RR)$ be a signal. Then 
\begin{equation}
\int_{\RR}\abs{V_f[q](\w,x)}^2 dx = \int_{\RR} \abs{\hq(z)}^2\abs{\hf(z-\w)}^2  dz.
\label{eq:Vfs_slice_STFT}
\end{equation}
\end{lemma}

The following lemma will be used in the proofs of Propositions \ref{prop:LVD_STFT} and \ref{Other_class_LTFT}.
\begin{lemma}
\label{lem:STFT_LVD2}
Let $f\in L^2(\RR)$ be supported in $[-S,S]$ and satisfy (\ref{eq:fSTFT_decay2}). Let $q\in V_{M,R}$, with $R=O(M)$. Then for every $W>4$,
\begin{equation}
    \label{eq:Lem_LTFT_STFT}
    \int_{[-WM/R,WM/R]^c}\int_{\RR} \abs{\hq(z)}^2\abs{\hf(z-\w)}^2  dzd\w = o_W(1)\norm{q}_2^2,
\end{equation}
where $[-WM/R,WM/R]^c$ is the set $\{\w\in\RR\ |\ w\notin[-WM/R,WM/R]\}$, and $o_W(1)$ is a function that decays to zero as $W\rightarrow\infty$.
\end{lemma}

\begin{proof}
We consider $\w>0$ and $z>0$, and note that the other cases are similar.
For each value of $\w>WM/R$, we decompose the integral (\ref{eq:Vfs_slice_STFT}) along $z$ into the two integrals in 
\[z\in(0,(W^{1/2}M/R+\w)/2)\quad  {\rm and} \quad z\in((W^{1/2}M/R+\w)/2,\infty).\]
For $z\in(0,(W^{1/2}M/R+\w)/2)$, 
since $\w\geq MW/R$ and $z\leq(W^{1/2}M/R+\w)/2$, 
\[z-\w \leq( W^{1/2}M/R-\w)/s \leq -\frac{1}{2}(W-W^{1/2})M/R <0,\]
so $\abs{z-\w}^{-2\k}$ obtains its maximum at $z=(W^{1/2}M/R+\w)/2$.
Thus, by (\ref{eq:fSTFT_decay2}),
\[\int_{0}^{(W^{1/2}M/R+\w)/2} \abs{\hq(z)}^2\abs{\hf(z-\w)}^2  dz \leq \norm{q}_2^2  \max_{0 \leq z \leq (W^{1/2}M/R+\w)/2}C^{\prime 2}\abs{z-\w}^{-2\k}\]
\begin{equation}
=\norm{q}_2^2 C^{\prime 2}\abs{(W^{1/2}M/R+\w)/2-\w}^{-2\k}= \norm{q}_2^2 C^{\prime 2}\abs{(W^{1/2}M/R-\w)/2}^{-2\k}
\label{eq:STFTvolume20}
\end{equation}
Integrating the bound (\ref{eq:STFTvolume20}) for $\w\in(WM/R,\infty)$ gives
\begin{equation}
\begin{split}
\int_{WM/R}^{\infty}\int_{0}^{(W^{1/2}M/R+\w)/2} \abs{\hq(z)}^2\abs{\hf(z-\w)}^2  dz d\w &= (W-W^{1/2})^{1-2k}M^{1-2\k}R^{2\k-1}\norm{q}_2^2 O(1) \\
 & = o_W(1)(M/R)^{1-2\k} \norm{q}_2^2.
\end{split}
\label{eq:STFTvolume30}
\end{equation}
Note that $(M/R)^{1-2\k}=O(1)$ since $R=O(M)$ and $\k>1/2$.

For $z\in\big((W^{1/2}M/R+\w)/2,\infty\big)$, $\hq$ decays like $M^{1/2}(z-M)^{-1}$. 
Indeed, since $z>M$ 
\begin{equation}
\begin{split}
R^{1/2}\sum_{n=-M}^M c_n {\rm sinc}\big(R(z-R^{-1}n)\big) & \leq R^{-1/2}\norm{\{c_n\}}_2 \sqrt{\sum_{n=-M}^M \frac{1}{(z-R^{-1}n)^2} } \\
& \leq R^{-1/2}\norm{q}_2 \sqrt{\sum_{n=-M}^M \frac{1}{(z-M/R)^2} } \\
& 
\leq   2R^{-1/2}\norm{q}_2\sqrt{M} (z-M/R)^{-1}.
\end{split}
\label{eq:STFTvolume1}
\end{equation}
Now, by (\ref{eq:STFTvolume1}),
\begin{equation}
\begin{split}
 & \int_{(W^{1/2}M/R+\w)/2}^{\infty} \abs{\hq(z)}^2\abs{\hf(z-\w)}^2  dz \\
 & \leq 4R^{-1}\norm{f}_2^2\norm{q}_2^2 M \max_{(W^{1/2}M/R+\w)/2 \leq z < \infty} (z-M/R)^{-2}\\
& =16R^{-1}\norm{f}_2^2\norm{q}_2^2  M \big(\w+(W^{1/2}-2)M/R\big)^{-2}.
\end{split}
\label{eq:STFTvolume2}
\end{equation}
Integrating the bound (\ref{eq:STFTvolume2}) for $\w\in(WM/R,\infty)$ gives
\begin{equation}
\int_{WM/R}^{\infty}\int_{(W^{1/2}M/R+\w)/2}^{\infty} \abs{\hq(z)}^2\abs{\hf(z-\w)}^2  dz d\w = (W+W^{1/2}-2)^{-1}\norm{q}_2^2 O(1) .
\label{eq:STFTvolume3}
\end{equation}
The bounds (\ref{eq:STFTvolume30}) and (\ref{eq:STFTvolume3}) give together (\ref{eq:Lem_LTFT_STFT}).
\end{proof}

\begin{proof}[Proof of Proposition \ref{prop:LVD_STFT}]
Let $q\in V_{M,R}$
Lemmas \ref{lem:STFT_LVD1} and \ref{lem:STFT_LVD2},
 are combined to give
$\norm{(I-\psi^W_{M,R})V_f[q]}_2= o_W(1)\norm{q}_2$, 
so by the frame inequality
\[\frac{\norm{(I-\psi^W_{M,R})V_f[q]}_2}{\norm{V_f[q]}_2}= o_W(1).\]
This means that given $\e>0$, we may choose $W$ large enough to guarantee $\frac{\norm{(I-\psi^W_{M,R})V_f[q]}_2}{\norm{V_f[q]}_2} < \e$, 
 and also guarantee that for every $R,M\in\NN$, 
$\norm{\psi_{M,R}}_1\leq C_{\epsilon} M$, 
with $C_{\epsilon}= 2W(1+2S)$, by (\ref{eq:VolGM_CWT2}), for $R=O(M)$.
\end{proof}

%

The proof of Proposition \ref{Other_class_LTFT} is similar to that of  Proposition \ref{prop:LVD_STFT}. We start with an analogous lemma to Lemma \ref{lem:STFT_LVD1}, based on (\ref{eq:LTFT_atom}) and Lemma \ref{Transform_lemma}.
\begin{lemma}
\label{lem:STFT_LVD3}
Let $V_f$ be the LTFT (Definition \ref{The localizing time-frequency continuous frame}), and let $q\in L^2(\RR)$ be a signal. Then, for any $(\w,\tau)$ such that $\abs{\w}>b^{M,R}_{\tau}$.
\begin{equation}
\int_{\RR}\abs{V_f[q](x,\w,\tau)}^2 dx = \int_{\RR} \abs{\hq(z)}^2\abs{ [\mathcal{D}(b^{M,R}_{\tau}/\tau)\hf](z-\w)}^2  dz.
\label{eq:Vfs_slice_LTFT}
\end{equation}
\end{lemma}

\begin{proof}[Proof of Proposition \ref{Other_class_LTFT}]

Let $\e>0$.
Note that since $V_f[s_{M,R}]$ is supported in the $x$ direction in $[R/2- \tau_2/a^{M,R}_{\tau}, R/2+ \tau_2/a^{M,R}_{\tau}]$, it is enough to show that restricting the frequency direction of $G$ to $ -WM/ R < \w < WM/ R$ results in an error less than $\e$.
Note that all of the truncated atoms, with $\abs{\w}\geq WM/ R$, are high frequency STFT atoms. 

By Lemma \ref{lem:STFT_LVD3}, we must study the error
\begin{equation}
    \label{eq:LTFT_LVD_e}
    \int_{\tau_1}^{\tau_2} \int_{[-WM/R,WM/R]^c}\int_{\RR} \abs{\hs_{M,R}(z)}^2\abs{ [\mathcal{D}(b^{M,R}_{\tau}/\tau)\hf](z-\w)}^2  dz d\w d\tau.
\end{equation}
Note that all atoms $[\mathcal{D}(b^{M,R}_{\tau}/\tau)\hf]$ in (\ref{eq:LTFT_LVD_e}) satisfy (\ref{eq:fSTFT_decay2}) with some constant $Y',C''$ instead of $Y,C'$. Hence, 
by Lemma \ref{lem:STFT_LVD2}, the error (\ref{eq:LTFT_LVD_e}) is of order $o_W(1)\norm{s_{M,R}}^2_2$. The rest of the proof is the same as the proof of Proposition \ref{prop:LVD_STFT}.

\end{proof}

\subsection*{Acknowledgements}
Ron Levie was partially supported by the DFG Grant DFG SPP 1798 “Compressed Sensing in Information Processing.” Haim Avron was partially supported by BSF grant 2017698.

%

\end{document}